\documentclass[11pt]{amsart}
\usepackage[english]{babel}
\usepackage[T1]{fontenc}
\usepackage{geometry}

\usepackage{fullpage}

\usepackage{amsmath}
\usepackage{amsfonts}
\usepackage{amssymb}
\usepackage{amsthm}
\usepackage{amsxtra}
\usepackage{amsopn}
\usepackage{amscd}
\usepackage{pifont}
\usepackage{latexsym}
\usepackage{verbatim}
\usepackage{pb-diagram}
\usepackage{cancel}

\newtheorem{prop}{Proposition}[section]
\newtheorem{cor}{Corollary}[section]
\newtheorem{thm}{Theorem}[section]

\theoremstyle{remark}
\newtheorem{oss}{Remark}[section]

\theoremstyle{definition}
\newtheorem{defn}{Definition}[section]

\newcommand{\gt}{{g}_\tau}
\newcommand{\ads}{\mathbb{H}^3_1(\kappa / 4)}
\newcommand{\adsk}{\mathbb{H}^3_1 (\kappa)}
\newcommand{\adsl}{\mathbb{H}^3_{1, \tau}}
\newcommand{\nt}{{\nabla}^\tau}
\newcommand{\Rt}{{R}^\tau}
\newcommand{\lm}{\lambda}
\newcommand{\nq}{\nu^2}
\newcommand{\rad}{\sqrt{\lm+\nu^2}}
\newcommand{\tb}{\tilde{b}}
\newcommand{\ta}{\tilde{a}}
\newcommand\B{{\mathbb B}} 
\newcommand\R{{\mathbb R}} 
\newcommand{\Rqd}{{\mathbb R}^4_2}
\newcommand{\sk}{\sqrt{\kappa}}
\newcommand{\kq}{\frac{\kappa}{4}}
\newcommand{\skm}{\frac{\sk}{2}}
\newcommand{\dskm}{\dfrac{\sk}{2}}
\newcommand{\tq}{\tau^2}

\newcommand{\radkm}{\dfrac{\sqrt{\kappa(\lm+\nu^2)}}{2}}
\newcommand{\sluu}{\sqrt{ w_{11}}}
\newcommand{\sltt}{\sqrt{-w_{33}}}

\newcommand{\cst}{\mathrm{constant}}

\def \r{\mbox{${\mathbb R}$}}

\def \s{\mbox{${\mathbb S}$}}

\usepackage{fullpage}

\title{Helix surfaces for Berger-like metrics \\ on the anti-de Sitter space}
\author{Giovanni Calvaruso}
\address{Dipartimento di Matematica e Fisica \lq\lq E. De Giorgi\rq\rq, Universit\`a del Salento, Prov. Lecce-Arnesano, 73100 Lecce, Italy}
\email{giovanni.calvaruso@unisalento.it}

\author{Irene I. Onnis}
\address{Dipartimento di Matematica e Informatica, Università degli Studi di Cagliari, Via Ospedale 72,
09124 Cagliari, Italy}
\email{irene.onnis@unica.it}

\author{Lorenzo Pellegrino}
\address{Dipartimento di Matematica e Fisica \lq\lq E. De Giorgi\rq\rq, Universit\`a del Salento, Prov. Lecce-Arnesano, 73100 Lecce, Italy}
\email{lorenzo.pellegrino@unisalento.it}

\author{Daria Uccheddu}
\address{Dipartimento di Matematica e Informatica, Università degli Studi di Cagliari, Via Ospedale 72,
09124 Cagliari, Italy}
\email{daria.uccheddu@unica.it}

\date{}
\thanks{The second and the fourth author were supported by a grant of Fondazione di Sardegna (Project GoAct).  The second author was also supported by the Thematic Project: Topologia Álgebrica,  Geométrica e Diferencial, Fapesp process number 2016/24707-4. The third author was supported by a scolarship of INdAM (Istituto Nazionale di Alta Matematica "Francesco Severi") reserved for master students in Mathematics}
\begin{document}

\begin{abstract}
We consider the Anti-de Sitter space $\mathbb{H}^3_1$ equipped with Berger-like metrics, that deform the standard metric of 
$\mathbb{H}^3_1$ in the direction of the hyperbolic Hopf vector field. Helix surfaces are the ones forming a constant angle with such vector field. After proving that these surfaces have (any) constant Gaussian curvature,  we achieve  their explicit local description in terms of a one-parameter family of isometries of the space and some suitable curves. These curves turn out to be  general helices, which meet at a constant angle the fibers of the hyperbolic Hopf fibration. 
\end{abstract}

\maketitle

\section{Introduction}
A {\em helix surface}  (or {\em constant angle surface}) is an oriented surface, whose normal vector field forms a constant angle with a
fixed field of directions in the ambient space. In recent years, many authors investigated helix surfaces in different ambient spaces.
Several examples of the study of helix surfaces in Riemannian settings may be found in \cite{CD}-\cite{DMVV},\cite{FMV},\cite{LM},\cite{MO},\cite{MOP},\cite{Ni} and references therein. 

The investigation of helix surfaces also extended to other settings. On the one hand, higher codimensional Riemannian helix surfaces were studied (see for example \cite{DR},\cite{DR2},\cite{Ru}). On the other hand, Lorentzian ambient spaces were considered. Lorentzian settings allow to more possibilities, as both spacelike and timelike surfaces can be studied. Some examples of the study of the geometry of helix surfaces in Lorentzian spaces are given in \cite{LM2},\cite{LO},\cite{OP},\cite{OPP}. In particular, helix surfaces of the anti-de Sitter space  $\mathbb{H}^3_1$ were studied in \cite{LO}. Equipping 
$\mathbb{H}^3_1$ with its canonical metric of constant curvature, all left-invariant vector fields are Killing (see also \cite{CP}), so no special directions emerge. Moreover, all helix surfaces turn out to be flat \cite{LO}. 

In \cite{CP}, the first author and D. Perrone introduced and studied a new family  of metrics $\tilde g_{\lambda\mu\nu}$ on $\ads$. These metrics were induced in a natural way by corresponding metrics defined on the tangent sphere bundle $T_1 \mathbb H^2 (\kappa)$, after  describing the covering map $F$ from $\mathbb H_1 ^3 (\kappa/4)$  to  $T_1 \mathbb H ^2(\kappa)$  in terms of paraquaternions. A crucial role in this construction is played by the {\em hyperbolic Hopf map} and the {\em hyperbolic Hopf vector field}, that is, the hyperbolic counterparts of the Hopf map and vector field on $\mathbb S^3$, respectively. 

 In \cite{MO},\cite{MOP},\cite{OP} and \cite{OPP},  the second author et al.  gave an explicit local classification of helix surfaces by means a one-parameter family of isometries of the ambient space and a suitable curve.   A similar appproach will be used in this paper in order to find the explicit characterization of the helix surfaces in the anti-de Sitter space.  We shall equip the anti-de Sitter space $\ads$ with a special type of the Lorentzian metrics introduced in \cite{CP}, that is, the ones that deform the standard metric of $\ads$ only in the direction of the hyperbolic Hopf vector field, which is then a Killing vector field.  Because of their analogies with the Berger metrics on $\mathbb S^3$, we shall refer to these metrics as {\em Berger-like metrics}. We shall consider the anti-de Sitter space $\ads$ equipped with Berger-like metrics and completely describe their surfaces, whose normal vector field forms a constant angle with the hyperbolic Hopf vector field. 

The paper is organized in the following way. In Section~2 we provide some needed information on the anti-de Sitter space $\ads$. 
In Section~3 we describe the Levi-Civita connection and the curvature of Berger-like metrics $\gt$ on $\ads$. The general equations for a surface of $(\ads,\gt)$ are given in Section~4 and then are applied in Section~5 to the case of helix surfaces. Three different cases occur, according to the sign of some constant $B$, which depends on the parameter $\tau$ of the Berger-like metric, the causal character of the surface and the constant angle. These three cases are completely described and characterized  in Section~6, in terms of a $1$-parameter family of isometries on $\adsl$ and some suitable curves. As we prove in Section~7,  these curves are general helices, which meet at a constant angle the fibers of the hyperbolic Hopf fibration.

\section{Preliminaries}
Let $\mathbb R^4_2 = (\mathbb R^4, g_0)$ denote the $4$-dimensional pseudo-Euclidean space, of neutral signature $(2,2)$, equipped with the flat pseudo-Riemannian  metric
$$g_0=dx_1^2+dx_2^2-dx_3^2-dx_4^2.$$ 
For any real constant $\kappa>0$,  the {\em anti-de Sitter (three)-space}  $\adsk$ 
is the hypersurface of $\mathbb R^4_2$ defined by
$$\mathbb H^3_1 (\kappa)=\{(x_1,x_2,x_3,x_4) \in \mathbb R^4 :  x_1^2 +x_2^2-x_3^2-x_4^2=-1/\kappa \}.$$
We denote by $\langle \,  ,\,\rangle$ the Lorentzian metric induced from $g_0$ on $\adsk$. This metric, known as the canonical metric of $\mathbb H^3_1$, has constant sectional curvature $-\kappa <0$. 

Consider now the algebra $\B$ of paraquaternionic numbers over $\R$ generated by $\{1,i,j,k\}$, where $k=ij$, $-i^2=j^2=1$ and $ij=-ji$. This is an associative, non-commutative and unitary algebra over $\R$. An arbitrary paraquaternonic number is given by $q=x_1+x_2i+x_3j+x_4k$. The conjugate of $q$ is $\bar q=x_1-x_2i-x_3j-x_4k$ and the norm of $q$ is given by
$$
||q||^2=q\bar q = x_1^2+x_2^2-x_3^2-x_4^2.
$$
The scalar product induced by such norm on $\R^4$  is exactly $g_0$, and $\{1,i,j,k\}$ is a pseudo-orthonormal basis with $1,i$ spacelike and $j,k$ timelike. In terms of paraquaternionic numbers, we have
$$
\adsk=\{q\in \B: ||q||^2=-1/\kappa \}.
$$
The above presentation was used in  \cite{CP} to describe a covering map
from the anti-de Sitter space $\mathbb H_1 ^3 (\kappa/4)$ to the unit tangent sphere bundle of the (Riemannian) hyperbolic two-space 
$\mathbb H ^2(\kappa)$, which is embedded in $\mathbb R_1 ^3$ as
$$\mathbb H^2 (\kappa)=\{(x_1,x_2,x_3) \in \R_1 ^3 : x_1^2+x_2^2 - x_3^2=-1/\kappa, x_3 >0\}.
$$
Now consider each paraquaternonic number $q$ as a pair of complex number in the following way $q:=(z,w)=(x_1+ix_2,x_3+ix_4)$, we obtain the covering map
$$\label{effe}  
\begin{array}{rcl}
F : \mathbb H^3_1 (\kappa /4) &\to & T_1 \mathbb H ^2(\kappa) \\[4pt]
  (z,w)  &\mapsto & \Big( \dfrac{\sqrt{\kappa}}{4}\left(2 z\bar w , |z|^2 +|w|^2\right) ,  
	-\dfrac{\kappa}{4}\left(({z}^2 +{\bar w}^2), 2 {\rm Re}(z\, w) \right)   \Big). 
\end{array}
$$
The composition of $F$ with the canonical projection
\begin{equation}
\begin{array}{rcl}
\pi:& T_1 \mathbb H ^2(\kappa)&\to  \mathbb H ^2(\kappa), \\
 &(x,u)&\mapsto x
\nonumber
\end{array}
\end{equation}
yields the {\it  hyperbolic Hopf map} $h= \pi\circ F$, that is,
\begin{equation}\label{Hopf}  
\begin{array}{rcl}
    h:     \ads &\to & \mathbb{H}^2(\kappa)\\[4pt]
               (z,w) & \mapsto & \dfrac{\sqrt{\kappa}}{4}(2 z\bar{w}, |z|^2+|w|^2).
\end{array}
\end{equation}
In particular, $h$ is a submersion with geodesic fibers, which can be defined as orbits of the $\mathbb S^1$-action   $\big(e^{it}, (z,w)\big)\mapsto \big(e^{it} z,e^{it} w)$ on the anti-de Sitter space $\mathbb H_1 ^3 (\kappa/4)$.  The vector field
\begin{align*}
    X_1(q)= \frac{\sqrt{\kappa}}{2}  iq =\frac{\sqrt{\kappa}}{2} (-x_2,x_1,-x_4,x_3)
\end{align*}
is tangent to fibers of  $h$ and satisfies $\langle X_1, X_1 \rangle=-1$ and is called the \textit{hyperbolic Hopf vector field}, in analogy with the case of the Hopf vector field and the Hopf fibration. We also consider 
\begin{align*}
    X_2(q)= \frac{\sqrt{\kappa}}{2}  jq = \frac{\sqrt{\kappa}}{2} (x_3,-x_4,x_1,-x_2), \quad X_3(q)= \frac{\sqrt{\kappa}}{2} kq = \frac{\sqrt{\kappa}}{2}(x_4,x_3,x_2,x_1).
\end{align*}
Then, $\{X_1,X_2, X_3\}$ parallelize $\ads$ and  
\begin{align*}
    \langle X_2, X_2 \rangle=\langle X_3, X_3 \rangle =1, \qquad  \langle X_l, X_m \rangle =0, \quad \text{for} \; l\neq m.
\end{align*}
\section{Levi-Civita connection and curvature}
In \cite{CP}, the covering map $F$ described by \eqref{effe}  was used to introduce a  family of Lorentzian metrics on the anti-de Sitter space $\ads $. These metrics were referred to as {\em metrics of Kaluza-Klein type},  because they correspond in a natural way to metrics defined on the unit tangent sphere bundle $T_1 \mathbb H ^2 (\kappa) $. Their general description is given by
$$
g=-\rho\, \theta^1 \otimes \theta^1 + \mu \,\theta^2 \otimes \theta^2+\nu \,\theta^3 \otimes \theta^3 ,
$$
where $\rho,\mu,\nu$ are positive real numbers and $\{\theta^i \}$ denotes the basis of $1$-forms dual to $\{X_i\}$. 

Here we shall focus our attention on a special type of metrics of Kaluza-Klein type, requiring to have a deformation only in the direction of the hyperbolic Hopf vector field $X_1$, which in this case is a Killing vector field \cite{CP}. Up to homotheties, we may restrict to the case where $\mu=\nu=1$ and $\rho=\tq$ and consider the one-parameter family of metrics
\begin{align}\label{31}
    \gt=-\tq\, \theta^1 \otimes \theta^1 +\theta^2 \otimes \theta^2+\theta^3 \otimes \theta^3, \quad \tau > 0 .
\end{align}
These metrics can be clearly considered an hyperbolic analogue of {\em Berger metrics} on the sphere $\mathbb S^3$ and so we refer to them as {\em Berger-like metrics} on $\ads$.
Observe that with respect to the standard metric $\langle \,,\,\rangle$ on $\ads$, these metrics can be described as
\begin{align*}
    \gt(X,Y)=\langle X,Y \rangle  + (1-\tq) \langle X,X_1 \rangle \langle Y,X_1 \rangle.
\end{align*}
{ and } $\{ E_1=\tau^{-1} X_1, E_2=X_2, E_3=X_3 \}$ is a pseudo-orthonormal basis for $\gt$, with $E_1$ timelike and $E_2, E_3$ spacelike. Computing the Lie brackets $[E_i,E_j]$, we find
\begin{align*}
 [E_1,E_2]=-\frac{\sk}{\tau}\,E_3, \qquad 
[E_2,E_3]=\tau\sk\,  E_1 , \qquad
[E_1,E_3]=\frac{\sk}{\tau}\,E_2 .
\end{align*}
Then, by the Koszul formula, we obtain the description of the Levi-Civita connection of $\adsl=(\ads, \gt)$ with respect to $\{E_1,E_2,E_3\}$. Explicitly, we get 
\begin{equation}\label{LCivita}
\begin{array}{lll}
    \nt_{E_1} E_1=0,    \quad &   \nt_{E_1} E_2= \dfrac{\sk \,(\tq-2)}{2\tau} E_3,  \quad & \nt_{E_1} E_3=-\dfrac{\sk\,(\tq-2)}{2\tau} E_2, \\[7pt]
    \nt_{E_2} E_1=\dfrac{\tau\sk }{2}E_3,   \quad  &   \nt_{E_2}E_2 = 0,  \quad & \nt_{E_2} E_3=\dfrac{\tau\sk }{2}E_1, \\[7pt]
    \nt_{E_3} E_1=-\dfrac{\tau\sk}{2} E_2,  \quad  &   \nt_{E_3} E_2= -\dfrac{\tau\sk }{2} E_1, \quad  & \nt_{E_3} E_3=0.
    \end{array}
\end{equation}
Observe that $E_1$ is an unit timelike vector field tangent to the fibers of $h$ and the Levi-Civita connection satisfies the following geometric relation
\begin{equation}\label{nxe1}
    \nt_X E_1=- \tau \skm X \wedge E_1, 
\end{equation}
for any tangent vector field $X$, where $E_1 \wedge E_2=E_3$, $E_2 \wedge E_3=-E_1$ e $E_3 \wedge E_1=E_2$ completely determine the wedge product $U \wedge V$ (see \cite{Ra}).  We now consider the curvature tensor, taken with the sign convention
\begin{equation*}
    \Rt(X,Y)=[\nt_X, \nt_Y]-\nt_{[X,Y]}.
\end{equation*}
Using \eqref{LCivita} we {have that the non vanishing components of the Riemannian curvature tensor $\Rt$ are}
\begin{align}\label{Rt}
    \begin{split}
    \Rt(E_1,E_2)E_1=-\kq \tq E_2, \quad
    \Rt(E_1,E_3)E_1=-\kq \tq E_3, \quad
    \Rt(E_3,E_2)E_3=\kq (4-3\tq)E_2
    \end{split}
\end{align}
and we can prove the following result.

\begin{prop}\label{curvatura}
The curvature tensor of $\adsl$ is given by
\begin{equation}
\begin{split}
    \Rt(X,Y)Z=&\kq(3\tq-4)[\gt(Y,Z)X-\gt(X,Z)Y]\\
    &+\kappa(\tq-1)[\gt(Y,E_1)\gt(Z,E_1)X-\gt(X,E_1)\gt(Z,E_1)Y\\
    &\hspace{55pt}+\gt(Y,Z)\gt(X,E_1)E_1-\gt(X,Z)\gt(Y,E_1)E_1],
\end{split}
\end{equation}
for all tangent vector fields $X,Y,Z$.
\end{prop}
\begin{proof}
Consider three arbitrary vector fields $X, Y, Z$ on $\adsl$ and their decompositions as 
\begin{align*}
    X=\bar{X}+xE_1, \quad Y=\bar{Y}+yE_1, \quad Z=\bar{Z}+zE_1,
\end{align*}
with $\bar{X}, \bar{Y}, \bar{Z}$ orthogonal to  $E_1$ and $x=\gt(X,E_1)$ and so on. Observe that using \eqref{Rt} we have that all terms of $\gt(\Rt(X,Y)Z,W)$ where 
$E_1$ occurs either one, three or four times necessarily vanish. Therefore, we have
\begin{align*}
\begin{split}
    \gt(\Rt(X,Y)Z,W)&=\gt(\Rt(\bar{X},\bar{Y})\bar{Z},\bar{W})\\
    &+yz\gt(\Rt(\bar{X},E_1)E_1,\bar{W})+xz\gt(\Rt(E_1,\bar{Y})E_1,\bar{W})\\
    &+wx\gt(\Rt(E_1,\bar{Y})\bar{Z},E_1)+wy\gt(\Rt(\bar{X},E_1)\bar{Z},E_1).
\end{split}
\end{align*}
From the decompositions of $\bar{X}$, $\bar{Y}$, $\bar{Z}$ and $\bar{W}$ follows that
\begin{align*}
    \begin{split}
        \gt(\Rt(\bar{X},\bar{Y})\bar{Z},\bar{W})=\kq(3\tq-4)(\gt(\bar{X},\bar{W})\gt(\bar{Y},\bar{Z})-\gt(\bar{X},\bar{Z})\gt(\bar{Y},\bar{W})).
    \end{split}
\end{align*}
Moreover, we have
\begin{align*}
    \Rt(\bar{X}, E_1)E_1=\kq\tq \bar{X},\qquad     \Rt(E_1,\bar{Y})E_1=-\kq\tq \bar{Y}.
\end{align*}
Thus, we conclude that 
\begin{align*}
  \gt(\Rt(X,Y)Z,W)=&\kq(3\tq-4)(\gt(\bar{X},\bar{W})\gt(\bar{Y},\bar{Z})-\gt(\bar{X},\bar{Z})\gt(\bar{Y},\bar{W}))\\
&+\kq\tq(yz\gt(\bar{X}, \bar{W})-xz\gt(\bar{Y},\bar{W})+ wx\gt(\bar{Y}, \bar{Z})-yw\gt(\bar{X},\bar{Z}))\\
=&\gt \Big(\kq(3\tq-4)[\gt(\bar{Y},\bar{Z})\bar{X}-\gt(\bar{X},\bar{Z})\bar{Y}]\\
&+\kq\tq[\gt(Y, E_1)\gt(Z, E_1)\bar{X}-\gt(X,E_1)\gt(Z,E_1)\bar{Y}\\
&+\gt(X,E_1)\gt(\bar{Y}, \bar{Z})E_1-\gt(Y,E_1)\gt(\bar{X},\bar{Z})E_1], \bar{W}\Big)\\
=&\gt\Big (\kq(3\tq-4)[\gt({Y},{Z}){X}-\gt({X},{Z}){Y}]\\
&+\kappa(\tq-1)[\gt(Y, E_1)\gt(Z, E_1){X}-\gt(X,E_1)\gt(Z,E_1){Y}\\
&+\gt(X,E_1)\gt({Y},{Z})E_1-\gt(Y,E_1)\gt({X},{Z})E_1], W \Big ),
\end{align*}
which ends the proof since $W$ is arbitrary.
\end{proof}
We end this section describing the isometries of $\adsl$. Following the idea used in \cite{MOP} and \cite{OPP}, we observe that the isometry group of $\adsl$ is the four-dimensional indefinite unitary group $\mathrm{U}_1(2)$, that can be identified with
$$
\mathrm{U}_1(2)=\{A\in \mathrm{O}_2(4)\vert AJ_1=\pm J_1A\},
$$
where $J_1$ is the complex structure of $\r^4$  corresponding to $i$, i.e., defined by
\begin{equation}\label{J_1}
%J_1 = \left(\begin{matrix}J & 0 \\ 0 & J\end{matrix}\right),\quad J = \left(\begin{matrix}0 & -1 \\ 1 & 0\end{matrix}\right),
J_{1} =\begin{pmatrix}
 0&-1 & 0 & 0 \\
 1&0 & 0 & 0 \\
 0& 0 & 0 & -1 \\
  0 &0 & 1 & 0 \\
 \end{pmatrix}
\end{equation}
while
$$
\mathrm{O}_2(4)=\{A\in \mathrm{GL}(4,\mathbb{R})\vert A^t=\epsilon\,A^{-1}\,\epsilon\},\qquad \epsilon=\begin{pmatrix}I&0\\0&-I\end{pmatrix},\quad I=\begin{pmatrix}1&0\\0&1\end{pmatrix},
$$
is the pseudo-orthogonal group, i.e., the group of $4\times 4$ real matrices preserving the semi-definite inner product of $\r_2^4$.

We now consider a $1$-parameter family $A(v), v\in(a,b)\subset\mathbb{R}$, consisting of $4\times 4$ pseudo-orthogonal matrices commuting (anticommuting, respectively) with $J_1$. In order to describe explicitly the family $A(v)$, we consider the two product structures $J_2$ and $J_3$ of $\r^4$ corresponding to $j$ and $k$ respectively, that is,
$$
 { J_{2} =\begin{pmatrix}
 0&0 & 1 &0 \\
 0&0 & 0 & -1 \\
 1&0& 0 & 0 \\
0&-1 & 0 & 0 \\
 \end{pmatrix},\qquad
 J_{3} =\begin{pmatrix}
 0&0 & 0 & 1 \\
 0&0 & 1 & 0 \\
 0&1& 0 & 0 \\
  1&0 & 0 & 0 \\
 \end{pmatrix}.}
 $$
Since $A(v)$ is a pseudo-orthogonal matrix, the first row must be a unit vector ${\mathbf r}_1(v)$ of $\r^4_2$ for all $v\in(a,b)$. Thus, without loss of generality, we can take
$$
{\mathbf r}_1(v)=(\cosh\xi_1(v)\cos\xi_2(v), -\cosh\xi_1(v)\sin\xi_2(v), \sinh\xi_1(v)\cos\xi_3(v),-\sinh\xi_1(v)\sin\xi_3(v)),
$$
for some real functions $\xi_1,\xi_2$ and $\xi_3$ defined in $(a,b)$. Since $A(v)$ commutes  (anticommutes, respectively) with $J_1$ the second row of $A(v)$ must be ${\mathbf r}_2(v)=\pm J_{1}{\mathbf r}_1(v)$. Now, the four vectors $\{{\mathbf r}_1, J_1{\mathbf r}_1, J_2{\mathbf r}_1,J_3{\mathbf r}_1\}$ form a pseudo-orthonormal basis of $\r^4_2$, thus the third row ${\mathbf r}_3(v)$ of $A(v)$ must be a linear combination of them. Since ${\mathbf r}_3(v)$ is unit and it is orthogonal to both ${\mathbf r}_1(v)$ and $J_1{\mathbf r}_1(v)$, there exists a function $\xi(v)$ such that
$$
 {{\mathbf r}_3(v)=\sin\xi(v) J_2 {\mathbf r}_1(v)+\cos\xi(v) J_3 {\mathbf r}_1(v).}
$$
Finally, the fourth row of $A(v)$ is ${ {\mathbf r}_4(v)=\pm J_1{\mathbf r}_3(v)=\mp\sin\xi(v) J_2 {\mathbf r}_1(v)\pm\cos\xi(v) J_3 {\mathbf r}_1(v)}$.
This means that any $1$-parameter family $A(v)$ of $4\times 4$ pseudo-orthogonal matrices commuting (anticommuting, respectively) with $J_1$ can be described by four functions $\xi_1,\xi_2,\xi_3$ and $\xi$ as
\begin{equation}\label{eq-descrizione-A}
A(\xi,\xi_1,\xi_2,\xi_3)(v)=
\begin{pmatrix}
{\mathbf r}_1(v)\\
\pm J_1{\mathbf r}_1(v)\\
 {\sin\xi(v) J_2 {\mathbf r}_1(v)+\cos\xi(v) J_3 {\mathbf r}_1(v)}\\
 {\mp\sin\xi(v) J_2 {\mathbf r}_1(v)\pm\cos\xi(v) J_3 {\mathbf r}_1(v)}
\end{pmatrix}.
\end{equation}

\section{Structure equations for surfaces in $\adsl$}

Consider a pseudo-Riemannian oriented surface $M$ immersed in $\adsl$. Let $N$ denote the unit vector field normal to $M$ in 
$\adsl$. We set $\lm:=\gt(N,N)$, so that $M$ is spacelike if $\lm=-1$ and timelike if $\lm=1$.

{In the following, we compute} the Gauss and Codazzi equations for $M$, using the metric induced by $\gt$ on $M$, the shape operator $A$, the tangent projection of $E_1$ on $M$ and the \textit{angle function} $\nu:=\gt(N,E_1)\gt(N,N)=\lm\gt(N,E_1)$. The vector field $E_1$ decomposes as
\begin{align*}
    E_1=T+\nu N ,
\end{align*}
where $T$ is tangent to $M$, whence,
\begin{align}\label{gt}
    \gt(T,T)=-(1+\lm\nq).
\end{align}
{Denoting by $X$ and $Y$ two} vector fields tangent to $M$ {we have}, 
$$%\begin{itemize}
   % \item $
	\nt_XY=\nabla_XY+\alpha(X,Y), \qquad \nt_X N=-A(X),
$$%\end{itemize}
where $\nabla$ is the Levi-Civita connection of $M$ and $\alpha$ the second fundamental form with respect to the immersion in $\adsl$. Thus, we conclude that
\begin{align*}
    \alpha(X,Y)=\lm\,\gt(\alpha(X,Y), N)\, N=\lm\,\gt(\nt_XY,N)N =-\lm\,\gt(Y,\nt_XN)N=\lm\,\gt(Y,A(X))N.
\end{align*}
Observe that
\begin{equation}
    \begin{split}
        \nt_XE_1=&\nt_XT+X(\nu)N+\nu\nt_XN\\
        =&\nabla_XT+\alpha(X,T)+X(\nu)N-\nu A(X)\\
        =&\nabla_XT+\lm\,\gt(T,A(X))N+X(\nu)N-\nu A(X). 
    \end{split}
\end{equation}
On the other hand, by \eqref{nxe1} we have
\begin{equation}
    \begin{split}
        \nt_XE_1=&-\skm\tau X \wedge E_1=-\skm\tau X \wedge (T+\nu N)\\
        =&-\skm\tau (X\wedge T)-\skm\nu \tau (X\wedge N)\\
        =&-\skm\lm\tau \gt(X\wedge T, N)N+\skm\nu \tau (N\wedge X)\\
        =&-\skm\lm\tau \gt(JX, T)N+\skm\nu \tau\, JX,
    \end{split}
\end{equation}
where $JX:=N\wedge X$ satisfies
\begin{equation}\label{JX}
\gt\big(JX,JY\big)=-\lambda\,\gt(X,Y), \qquad J^2X=\lambda\, X.
\end{equation} 
Comparing the above expressions for $\nt_XE_1$, we find 
\begin{equation}\label{Xna}
\left\{\begin{aligned}
    \nabla_XT&=\nu\,(A(X)+\skm\tau JX), \\
    X(\nu)&=-\lm \,\gt(A(X)+\skm \tau JX, T).
    \end{aligned}
    \right.
\end{equation}
We now prove the following.

\begin{prop}\label{gc}
{Denoting by $X$ and $Y$ two} vector fields tangent to $M$, with $K$ the Gaussian curvature of $M$ and with $\bar{K}$ the sectional curvature in $\adsl$ of the plane tangent to $M$, we have
%Con la notazione impiegata sinora valgono le seguenti:
\begin{equation}\label{kk}
    K=\bar{K}+\lm\, \mathrm{det}\, A=-\kq\tq +\lm\,[\mathrm{det}\,A+\kappa\nu^2(\,1-\tq)]
\end{equation}
and
\begin{equation}\label{Cod}
  \nabla_XA(Y)-\nabla_YA(X)-A[X,Y]=-\kappa \lm\nu(1-\tq)\,[\gt(X,T)Y-\gt(Y,T)X].
\end{equation}
\end{prop}
\begin{proof}
Recall that for a pseudo-Riemannian surface, one has
\begin{align*}
     K=\bar{K}+\lm\, \dfrac{\gt(A(X),X)\gt(A(Y),Y)-\gt(A(X),Y)^2}{\gt(X,X)\gt(Y,Y)-\gt(X,Y)^2}.
\end{align*}
Consider a local (pseudo-)orthonormal basis $\{ X,Y\}$ on $M$, i.e., $\gt(X,X)=1$, $\gt(X,Y)=0$ and $\gt(Y,Y)=-\lm$. Then, by Proposition~\ref{curvatura}, we have
\begin{align*}
    \Rt(X,Y)Y=&\kq(3\tq-4)[\gt(Y,Y)X-\gt(X,Y)Y]\\
    &+\kappa\,(\tq-1)[\gt(Y,E_1)\gt(Y,E_1)X-\gt(X,E_1)\gt(Y,E_1)Y\\
    &+\gt(Y,Y)\gt(X,E_1)E_1-\gt(X,Y)\gt(Y,E_1)E_1]
\end{align*}
and so,
\begin{align*}
    \gt(\Rt(X,Y)Y,X)=&-\kq\lm(3\tq-4)+\kappa(\tq-1)[\gt(Y,E_1)^2-\lm\gt(X,E_1)^2]\\
    =&-\kq\lm (3\tq-4)+\kappa(\tq-1)[\gt(Y,T)^2-\lm\gt(X,T)^2].
\end{align*}
Taking into account $\gt(X,T)^2-\lm\gt(Y,T)^2=\gt(T,T)=-(1+\lm\nq)$, we have
\begin{align*}
    \Rt(X,Y,Y,X)=-\kq\lm(3\tq-4)+\kappa\lm(\tq-1)(1+\lm\nq)
\end{align*}
that gives
$$
    \bar{K}=-\lm\Rt(X,Y,Y,X)=\kq(3\tq-4)-\kappa(\tq-1)(1+\lm\nq)=-\kq\tq+\kappa\lm\nq(1-\tq),
$$
whence equation \eqref{kk} follows.

Consider the Codazzi equation
$$
    \gt(\Rt(X,Y)Z,N)=\gt(\nabla_XA(Y)-\nabla_YA(X)-A[X,Y],Z).
$$
From Proposition \ref{curvatura}, we have
\begin{align*}
    \Rt(X,Y)N=&\kappa\,(\tq-1)\lm\nu[\gt(Y,E_1)X-\gt(X,E_1)Y]\\
    =& \kappa \lm\nu \,(\tq-1)[\gt(Y,T)X-\gt(X,T)Y],
\end{align*}
that leads equation \eqref{Cod} by the arbitrarity of $Z$.
\end{proof}

\section{Basic properties of helix surfaces in $\adsl$}

We start with the following.
\begin{defn}
Let  $M$ be an oriented pseudo-Riemannian surface immersed in $\adsl$ and $N$ the unit vector field normal to $M$, with $\gt(N,N)=\lm$.  The surface
$M$ is called a \textit{helix surface} (or \textit{a constant angle surface}) if the angle function $\nu=\lm\,\gt(N,E_1)$ is constant on  $M$.
\end{defn}

\begin{oss}\label{nun0}
If $M$ is spacelike (respectively, timelike), then $T$ is spacelike (respectively, timelike) and $JT$ is spacelike. 
	From \eqref{gt} and \eqref{JX} we get 
	$$
	\gt( T,T)=-(1+\lambda \, \nu^2), \qquad \gt( JT,JT)=\lambda+\nu^2>0,\qquad \gt( T,JT)=0.
	$$ 
Observe that in the case $\lm=1$,  if $\nu=0$ then $E_1$ is tangent to $M$ at each point. Therefore, $M$ is a  \textit{Hopf tube}.
On the other hand, if $\lm=-1$,  then $|\nu|>1$ and so $\nu\neq 0$.
\end{oss}
Taking into account Remark~\ref{nun0} from now on we assume $\nu\neq 0$. 

\begin{prop}\label{p4}
Let $M$ be a helix surface in $\adsl$ and $N$ the unit vector field normal to $M$. Then {the following hold:}
\begin{itemize}
    \item [(i)] with respect to the tangent basis  $\{T, JT\}$, the matrix of the shape operator is given by
   $$
        A=    \begin{pmatrix}
        0& -\dfrac{\sqrt{\kappa}}{2}\lm\tau\\[4pt]
       \dfrac{\sqrt{\kappa}}{2}\tau  & \mu
    \end{pmatrix},
   $$
		for some smooth function $\mu$ on $M$;
    \item[(ii)] the Levi-Civita connection $\nabla$ of $M$ is described by
		$$
		\begin{array}{ll}
        \nabla_TT=\nu\tau\sk \,JT, &  \nabla_{JT}T=\mu\nu \,JT,\\[4pt]
        \nabla_TJT=\lm\nu\tau \sk\, T, &  \nabla_{JT}JT=\lm\mu\nu \,T;
    \end{array}  
		$$
   \item[(iii)] the Gaussian curvature of $M$ is given by
	  $$
        K=\lm \kappa \nq (1-\tq);
$$
    \item[(iv)] the function $\mu$ satisfies equation
    \begin{equation}\label{star}
        T(\mu)+\nu\mu^2+\kappa\nu B=0,
    \end{equation}
    where $B:=\nq(\tq-1)-\lm$.
\end{itemize}
\end{prop}

\begin{proof}
Taking into account Remark~\ref{nun0} {for the tangent basis $\{T, JT \}$}  we have
\begin{align*}
    \gt(A(T),T)&=-\skm \tau \gt(J T,T)-\lm\, T(\nu)=0;\\
    \gt(A(J T),T)&=-\lm \, J T(\nu)-\skm \tau\gt(J^2T,T)=-\skm \lm\tau\gt(T,T),
\end{align*}
which, by the symmetry of the shape operator, yields  \textit{(i)}.

For the Levi-Civita connection $\nabla$ of $M$, {using~\eqref{Xna},} we have
\begin{align*}
    \nabla_TT&=\nu\Big(A(T)+\skm\tau \, JT\Big)=\sk\nu\tau\, JT;\\
    \nabla_{JT}T&=\nu \Big(A(JT)+\skm\tau  \, J^2T \Big)=\nu\mu\, JT,
\end{align*}
and using the compatibility of $\nabla$ with the metric $\gt$, follows
\begin{align*}
     \gt(\nabla_T JT, T)=-\gt(JT, \nabla_T T),\qquad \gt(\nabla_{JT} JT, T)=-\gt(JT, \nabla_{JT} T).
\end{align*}
So,
\begin{align*}
    \nabla_T JT =\sk\lm\nu \tau\, T, \qquad    \nabla_{JT} JT=\lm\nu\mu\, T.
\end{align*}

From \eqref{kk} we have that Gaussian curvature is given by
\begin{align*}
    K& =-\kq\tq+\lm\, [\mathrm{det} A+\kappa\nq(1-\tq)]\\
    &=\lm\kappa\nq\,(1-\tq).
\end{align*}

\begin{oss}\label{rem52}
It is worthwhile to remark that the Gaussian curvature $K$ is a constant, which depends on the causal character of the surface and on the sign of $(1-\tq)$ and it can assume any real value. In the special case of $\tau=1$ the Gaussian curvature vanishes and helix surfaces for the standard metric on the anti-de Sitter space are flat, coherently with the results obtained in \cite{LO}  for surfaces forming a constant angle with an unit Killing vector field.
\end{oss}

Finally, we calculate
\begin{align*}
    \nabla_TA(JT)&=-\frac{\kappa}{2}\lm\tq\nu \, JT+T(\mu)\, JT+ \sk\lm\tau\mu\nu\, T,\\
    \nabla_{JT} A(T)&=\frac{\sk}{2}\lm\tau \nu\mu \,T,\\
    A[T,JT]&=A(\nabla_TJT-\nabla_{JT}T)=\frac{\sk}{2}\lm\tau\nu\mu\, T-\mu^2\nu\, JT+\frac{\kappa}{2}\lm\tq\nu\, JT.
\end{align*}
By Proposition \ref{gc}, we  get
\begin{align*}
    \nabla_TA(JT)-\nabla_{JT}A(T)-A[T,JT]=&\kappa(\tq-1)\lm\nu(\gt(T,T)JT)\\
    =&-\kappa(\tq-1)\lm\nu(1+\lm\nq)JT
\end{align*}
and so,  by comparing,
\begin{align*}
    T(\mu)+\mu^2\nu-\kappa\lm\nu[1-\lm\nq(\tq-1)]=0
\end{align*}
which ends the proof.
\end{proof}

\noindent
We recall that $E_1$ is a timelike vector field and $\gt(E_1,N)=\nu\lm$. Thus, there exists a smooth function $\varphi$ on $M$, such that 
\begin{align*}
    N=-\nu\lm\, E_1+\rad \cos\varphi\, E_2 +\rad \sin\varphi\, E_3,
\end{align*}
whence,
\begin{align*}
    T=E_1-\nu N=(1+\nq\lm)\, E_1-\nu\rad \cos\varphi\, E_2 -\nu\rad \sin\varphi\, E_3
\end{align*}
and we can calculate 
\begin{align*}
    JT=N \wedge T=N\wedge (E_1-\nu N)=N\wedge E_1=\rad\,(\sin \varphi\, E_2-\cos\varphi\, E_3).
\end{align*}
Moreover, we have
\begin{align*}
    \nt_TE_1=&-\skm\tau\, T\wedge E_1=\skm\nu  \tau\, JT,\\[7pt]
    \nt_TE_2=&\skm (1+\nq\lm) \frac{\tq-2}{\tau}\,E_3+ \nu \tau\radkm  \sin \varphi\, E_1,\\[7pt]
     \nt_TE_3=&-\skm(1+\nq\lm)\frac{\tq-2}{\tau}\,E_2-\nu\radkm\tau \sin \varphi\, E_1.
\end{align*}
Therefore, 
$$
\begin{aligned}
    A(T) =&-\nt_T N  \\
	 	= &\nu\lm \nt_T E_1 -\rad \,[-\sin\varphi \,T(\varphi)\,E_2+\cos\varphi \,\nt_T E_2\\
		&+\cos\varphi \,T(\varphi)\, E_3+\sin \varphi\, \nt_T E_3]\\
    =&\skm\lm\nq \tau \,JT+T(\varphi)\,JT+\skm(1+\nq\lm) \,\frac{\tq-2}{\tau}\,JT.
\end{aligned}
$$
But $A(T)=\dskm \tau JT$ and so, we conclude that
\begin{align*}
\skm\lm\nq\tau +T(\varphi)+\skm(1+\lm\nq)\frac{\tq-2}{\tau}=\skm \tau,
\end{align*}
that is,
\begin{align*}
  T(\varphi)=-\frac{\sk}{\lm \tau}B.
\end{align*}
Moreover,
\begin{align*}
    &\nt_{JT}E_1=\rad(\sin\varphi \nt_{E_2} E_1-\cos\varphi \nt_{E_3}E_1)=\radkm\tau(\sin\varphi\, E_3+\cos\varphi\, E_2),\\[4pt]
    &\nt_{JT}E_2=-\rad \cos\varphi \nt_{E_3}E_2=\radkm\tau \cos\varphi\, E_1,\\[4pt]
     &\nt_{JT}E_3=\rad\tau \sin \varphi \nt_{E_2}E_3=\radkm\tau \sin \varphi\, E_1
\end{align*}
and so,
$$
\begin{aligned}
        A(JT)=&-\nt_{JT} N \\
			=&\nu\lm \nt_{JT} E_1-\rad\, [-\sin\varphi\, JT(\varphi)E_2+\cos\varphi \nt_{JT} E_2
    \\&+\cos\varphi\, JT(\varphi)\, E_3+\sin \varphi \nt_{JT} E_3]\\
    =&JT(\varphi)\,JT-\skm\lm \tau\, T.
\end{aligned}
$$
But $A(JT)=-\dskm\tau\lm\, T+\mu\, JT$, so that we obtain
\begin{align*}
    JT(\varphi)=\mu.
\end{align*}
In this way, we determined the following system 
$$
\left\{
\begin{aligned}   
   T(\varphi)&=-\dfrac{\sk}{\lm \tau}B,  \\[4pt]
    JT(\varphi)&=\mu,
\end{aligned}
\right.
$$
whose compatibility condition is given by 
\begin{align*}
    [T,JT]=\nabla_TJT-\nabla_{JT}T=\sk\nu\lm\tau\, T-\mu\nu\, JT.
\end{align*}
Recalling that 
\begin{align*}
    [T,JT](\varphi)=T(\mu),
\end{align*}
by comparison we obtain
\begin{align*}
    T(\mu)+\mu^2\nu+\kappa\nu B=0,
\end{align*}
that is, equation \eqref{star}, which we already established, so that the system is  compatible.

We can choose a system of local coordinates $(x,y)$ on $M$, such that
\begin{equation}\label{localsystem}
\left \{
\begin{aligned}
    \partial_x&=T,\\
    \partial_y&=a\,T+b\,JT,
\end{aligned}
\right.
\end{equation}
for some smooth functions $a=a(x,y)$, $b=b(x,y)$ on  $M$. Requiring that $[\partial_x, \partial_y]=0$, we get
$$
\left \{
\begin{aligned}
    a_x=&-\sk\lm\nu\tau b ,\\[3pt]
    b_x=&b\mu\nu.
    \end{aligned}
    \right .
    $$
So, equation \eqref{star} can be rewritten as follows:
\begin{equation}\label{diff}
    \mu_x=-\nu(\kappa B+\mu^2).
\end{equation}
In order to integrate equation \eqref{diff}, we need to consider separately the following cases
\begin{itemize}
    \item [(i)] If $B>0 $, we find
		\begin{align*}
   \mu(x,y)=\sqrt{\kappa B}\tan(\eta(y)-\nu\sqrt{\kappa B} \,x).
\end{align*}
    \item[(ii)] If $B=0 $, we have 
		\begin{align*}
   \mu(x,y)=\frac{1}{\nu x+\eta(y)}.
\end{align*}
    \item [(iii)]If $B<0 $, we get 
		\begin{align*}
   \mu(x,y)=\sqrt{-\kappa B}\tanh(\eta(y)+\nu\sqrt{-\kappa B}\,x). 
\end{align*}
\end{itemize}
In all the above cases,  $\eta(y)$ is an arbitrary smooth function.

As we are interested in only one coordinate system $(x,y)$ on the surface $M$, we only need one admissible solution for $a$ and $b$ in each case.

\subsection*{CASE $B>0$} Since $b_x=b\mu\nu$, we deduce
\begin{align*}
   b_x=b\sqrt{\kappa B}\tan(\eta(y)-\nu \sqrt{\kappa B}\, x)\,\nu ,
\end{align*}
which admits as solution $b=\cos(\eta(y)-\nu\sqrt{\kappa B}\, x)$. Then, we have
\begin{align*}
 a_x=-\lm\nu\sk\tau b  =-\sk\lm\nu\tau \cos(\eta(y)-\nu\sqrt{\kappa B}\, x)
\end{align*}
and so, we can take $a=\dfrac{\lm \tau}{\sqrt{B}}\sin(\eta(y)-\nu\sqrt{\kappa B}\, x)$. Now, from
$$
\left \{
\begin{aligned} 
    \varphi_x&=-\frac{\sk}{\lm\tau}B,\\[4pt]
    \varphi_y&=-\frac{a\sk}{\lm \tau}B+b\mu=0,   
\end{aligned}
\right .
$$
we get  $\varphi(x,y)=-\dfrac{\sk}{\lm\tau}Bx+c$, for some real constant $c$.

\subsection*{CASE $B=0$} From $b_x=b\mu\nu$, we now have
\begin{align*}
    b_x=b\frac{1}{\nu x+\eta(y)}\nu
\end{align*}
and a solution is given by $b=\nu x+\eta(y)$. Moreover, we have:
\begin{align*}
    a_x=-\sk\,\lm\nu\tau b=-\sk\,\lm\nu^2\tau x+\eta(y),
\end{align*}
which holds for $a=-\nu\skm\,\lm \tau  x (\nu x + 2\eta(y))$. Then,
$$
\left\{
\begin{aligned}   
    \varphi_x=&-\dfrac{\sk}{\lm\tau}B=0,\\[4pt]
    \varphi_y=&-\dfrac{a\sk}{\lm \tau}B+b\mu=1, 
\end{aligned}
\right .
$$
whence,  $\varphi(x,y)=y+c$ for some real constant $c$.

\subsection*{CASE $B<0$} Recalling that $b_x=b\mu\nu$, we have
\begin{align*}
    b_x=b\sqrt{-\kappa B}\tanh(\eta(y)+\nu\sqrt{-\kappa B} \,x)\,\nu,
\end{align*}
which is satisfied by $b=\cosh(\eta(y)+\nu\sqrt{-\kappa B}\, x)$. Moreover, 
we find:
\begin{align*}
    a_x=-\lm\nu\sk\,\tau b=-\lm\nu\sk\,\tau \cosh(\eta(y)+\nu\sqrt{-\kappa B}\, x)
\end{align*}
and so, we take $a=-\dfrac{\lm \tau}{\sqrt{-B}}\sinh(\eta(y)+\nu\sqrt{-\kappa B}\, x)$. Finally,
$$
\left \{
\begin{aligned} 
    \varphi_x&=-\dfrac{\sk}{\lm\tau}B, \\[4pt]
    \varphi_y&=-\frac{a\sk}{\lm \tau}B+b\mu=0,
\end{aligned}
\right .
$$
and we obtain $\varphi(x,y)=-\dfrac{\sk}{\lm\tau}Bx+c$, for some real constant  $c$.

 Using the above results,  we have the following. 
\begin{prop}\label{pepv}
Let $M$ be a helix surface in $\adsl$ with constant angle function $\nu$. With respect to the local coordinates $(x,y)$ defined above, the position vector $F$ of $M$ in $\mathbb{R}^4_2$ satisfies the following equation:
\begin{itemize}
\item [(a)] if $B=0$,
\begin{equation}\label{eqquarta1}
\frac{\partial^2F}{\partial x^2}=0,
\end{equation}
\item [(b)] if $B\neq 0$,
\begin{equation}\label{eqpv1}
\frac{\partial^4 F}{\partial x^4}+(\tb^2+2\ta)\frac{\partial^2 F}{\partial x^2}+\ta^2 F=0,
\end{equation}
where
$$
\ta=\kq\frac{B}{\tq}(\lm+\nq), \qquad \tb=-\sk\frac{B}{\lm \tau}.
$$
\end{itemize}
\end{prop}
\begin{proof}
Let $M$ be a helix surface and let
$$
F(x,y)=(F_1(x,y), F_2(x,y), F_3(x,y), F_4(x,y)).
$$
denote the position vector of $M$ in $\mathbb{R}^4_2$, described with respect to the local coordinates $(x,y)$ defined before. By definition of position vector, we get
\begin{align}\label{TdxF}
\partial_x F&=(\partial_x F_1, \partial_x F_2, \partial_x F_3, \partial_x F_4)=T\notag \\
&=\rad \left[ \lm \rad E_{1|F}-\nu \cos \varphi\, E_{2|F}-\nu \sin \varphi\, E_{3|F}\right].
\end{align}
Then, if we consider the expressions of $E_1, E_2$ and $E_3$ with respect to the coordinates of $\mathbb{R}^4_2$, we can  express the above equation as
\begin{align}\label{dexF}
    \begin{cases}
   \partial_x F_1=\radkm\left[ -\dfrac{\lm}{\tau} \rad\, F_2-\nu \cos \varphi \, F_3-\nu \sin \varphi \, F_4 \right],\\[7pt]
   \partial_x F_2=\radkm\left[ \dfrac{\lm}{\tau} \rad\, F_1+\nu \cos \varphi \, F_4-\nu \sin \varphi \, F_3 \right],\\[7pt]
   \partial_x F_3=\radkm\left[ -\dfrac{\lm}{\tau} \rad\, F_4-\nu \cos \varphi \, F_1-\nu \sin \varphi \, F_2 \right],\\[7pt]
   \partial_x F_4=\radkm\left[ \dfrac{\lm}{\tau} \rad\, F_3+\nu \cos \varphi \, F_2-\nu \sin \varphi \, F_1 \right].
    \end{cases}
\end{align}
Therefore, if $B = 0$ taking the derivative of \eqref{dexF} with respect to $x$,  we get \eqref{eqquarta1}.  

If we suppose $B\neq 0$,   we obtain
\begin{align} \label{dexxF}
    \begin{cases}
    (F_1)_{xx}=-\ta \, F_1-\tb\,(F_2)_x,\\[4pt]
    (F_2)_{xx}=-\ta \, F_2+\tb\,(F_1)_x,\\[4pt]
   	(F_3)_{xx}=-\ta \, F_3-\tb\,(F_4)_x,\\[4pt]
	(F_4)_{xx}=-\ta \, F_4+\tb\,(F_3)_x,
    \end{cases}
\end{align}
where
$$
\ta=\kq\frac{B}{\tq}(\lm+\nq), \qquad \tb=-\sk\frac{B}{\lm \tau}.
$$
In conclusion, taking twice the derivative of \eqref{dexxF} with respect to $x$ and using the previous relations we find \eqref{eqpv1}.
\end{proof}

\begin{oss}
From $|F|^2=-4/\kappa$ and relations~\eqref{dexF}, \eqref{dexxF}, we get:
\begin{equation}\label{norme}
\begin{array}{lll}
    \langle F, F \rangle=-\dfrac{4}{\kappa},    \quad &   \langle F_x, F_x \rangle=\dfrac{4}{\kappa}\ta,  \quad & \langle F, F_x \rangle=0, \\
       \langle F_x, F_{xx} \rangle=0,    \quad &   \langle F_{xx}, F_{xx} \rangle=D,  \quad & \langle F, F_{xx} \rangle=-\dfrac{4}{\kappa}\ta, \\[7pt]
       \langle F_x, F_{xxx} \rangle=-D,    \quad &   \langle F_{xx}, F_{xxx} \rangle=0,  \quad & \langle F, F_{xxx} \rangle=0,\\[7pt]
      \langle F_{xxx}, F_{xxx} \rangle=E, $ $ 
    \end{array}
\end{equation}
where
$$
D=\frac{4}{\kappa}(\ta \tb^2+3\ta^2), \qquad E=(\tb^2+2\ta)D-\frac{4}{\kappa}\ta^3.
$$
\end{oss}

\section{Characterization theorems for  helix surfaces in $\adsl$}

%\begin{oss}\label{rem-imm}
In order to give  conditions under which an immersion defines a helix surface in $\adsl$ we observe that, if $F$ is a position vector of a helix surface in $\adsl$ we have that:
$$
J_1F=\frac{2}{\sk}X_{1|F(x,y)}=\tau \frac{2}{\sk}E_{1|F(x,y)}
$$
and, using the \eqref{norme}, we have the following identities:
\begin{align}\label{IFnorme}
\begin{array}{ll}
\langle  J_1F, F_x \rangle = -\dfrac{2\lm(\lm+\nq)}{\tau \sk},\quad &
\langle  J_1F, F_{xx} \rangle = 0,\\[9pt]
\langle  J_1F_{xx}, F_x \rangle = \ta \left[\dfrac{2\lm(\lm+\nq)}{\tau \sk} -\dfrac{4}{\kappa} \tb \right]=L,\quad &
\langle  J_1F_{x}, F_{xxx} \rangle = 0,\\[9pt]
\langle  J_1F_{x}, F_{xx} \rangle + \langle  J_1F, F_{xxx} \rangle = 0, \quad &
\langle  J_1F_{xx}, F_{xxx} \rangle + \langle  J_1F_x, F_{xxxx} \rangle = 0.
\end{array}
\end{align}
%\end{oss}
%Using Remark~\ref{rem-imm} we can prove the following proposition, that gives the condition under which an immersion defines a helix surface.
We now use these relations to prove the following key result.
\begin{prop}\label{prop-vice}
Let $F: \Omega \rightarrow \adsl \subset \mathbb R^4_2$ be an immersion from an open set $\Omega\subset \mathbb{R}^2$, with local coordinates $(x,y)$ such that the projection of $E_1=\dfrac{\sk}{2\tau}J_1F$ is $F_x$. Then $F(\Omega)\subset \adsl$ defines a helix surface of constant angle function $\nu$ if and only if
\begin{align}\label{547}
     \gt (F_x,F_x)=\gt(E_1,F_x)=-\lm{(\lm+\nu^2)}
\end{align} 
and
\begin{align}\label{548}
     \gt (F_x,F_y)-\gt(F_y,E_1)=0.
\end{align}
\end{prop}
\begin{proof}
Suppose that $F(\Omega)$ is a helix surface in $\adsl$ of constant angle function $\nu$, then we have:
\begin{align*}
\gt(F_x,F_x)&=\langle F_x,F_x\rangle+(1-\tq)\langle F_x,X_1\rangle^2=\langle F_x,F_x\rangle+(1-\tq)\kq\langle F_x,J_1F\rangle^2\\
            &=\frac{4}{\kappa}\ta+(1-\tq)\left(-\dfrac{2\lm(\lm+\nq)}{\tau \sk}\right)^2=-\lm(\lm+\nq).
\end{align*}
In a similar way we find:
\begin{align*}
\gt(F_x,E_1)&=\frac{1}{\tau}\left[\langle F_x,X_1\rangle-(1-\tq)\langle F_x,X_1\rangle\right]\\
            &=\tau\langle F_x,X_1\rangle=-\lm(\lm+\nq),
\end{align*}
that leads to the equation \eqref{547}. In addition,  we have 
\begin{align*}
     \gt(F_y,E_1)=\gt(F_y,F_x+\nu N)=\gt(F_y,F_x)
\end{align*}
that is equation \eqref{548}.

To prove the converse, consider
\begin{align*}
            \tilde{T}=F_y-\frac{\gt(F_y,F_x)F_x}{\gt(F_x,F_x)}.
        \end{align*}
Then, we get the orthonormal basis $\{F_x, \tilde{T}, N\}$ for the tangent space to $\adsl$ along $F(\Omega)$.  Moreover,  from \eqref{547} and \eqref{548} we have:
        \begin{align*}
            \gt(E_1, \tilde{T})=\gt(E_1, F_y)-\gt(E_1, F_x)\frac{\gt(F_y,F_x)}{\gt(F_x,F_x)}=0.
        \end{align*}
This leads to $E_1=c_1F_x+c_2N$. Moreover from \eqref{547} we have $c_1=1$ that means $(E_1)^T=F_x$ and also
        \begin{align*}
            -1=\gt(E_1,E_1)=-(1+\lm\nu^2)+c_2^2\lm
        \end{align*}
whence, $\gt(E_1,N)=\lm \vert\nu\vert$ is constant and so, $F(\Omega)$ is a helix surface.
\end{proof}
In order to get explicit solutions of equations \eqref{eqpv1} and \eqref{eqquarta1} we consider three different cases, depending on the different possibilities for $B$.

\subsection{ Helix surfaces of $\adsl$ in the case $B>0$}
Integrating \eqref{eqpv1} we prove the following.

\begin{prop}\label{t1}
Let $M$ be a helix surface in $\adsl$ with constant angle function $\nu$ such that $B>0$. Then, with respect to the local coordinates $(x,y)$ defined above, the position vector $F$ of $M$ in $\mathbb{R}^4_2$ is explicitly given by
$$
F(x,y)=\cos(\alpha_1 x)w^1(y)+\sin(\alpha_1 x)w^2(y)+\cos(\alpha_2 x)w^3(y)+\sin(\alpha_2 x)w^4(y),
$$
where
$$
\alpha_{1,2}=\skm\Big(|\nu|\sqrt{B}\pm \frac{B}{\tau}\Big)
$$
are real constants and $w^i(y)$, $i=1,2,3,4$, are mutually orthogonal vector fields in $\mathbb{R}^4_2$, depending only on $y$, such that, setting $w_{ij}=\langle w^i(y), w^j(y) \rangle$ for all indices $i,j$, we have
\begin{equation}\label{w11andw33}
w_{11}= w_{22}=\frac{4 \tau}{\kappa^\frac{3}{2} B}\alpha_2, \qquad w_{33}= w_{44}= -\frac{4 \tau}{\kappa^\frac{3}{2} B}\alpha_1.
\end{equation}
\end{prop}
\begin{proof}
If $B>0$, then $\tb^2+2\ta>0$ and 
$\tb^2+4\ta>0$. Integrating equation \eqref{eqpv1}, we then obtain
\begin{equation*}
F(x,y)=\cos(\alpha_1 x)w^1(y)+\sin(\alpha_1 x)w^2(y)+\cos(\alpha_2 x)w^3(y)+\sin(\alpha_2 x)w^4(y)
\end{equation*}
where
$$
\alpha_{1,2}=\sqrt{\frac{\tb^2+2\ta\pm\sqrt{\tb^2(\tb^2+4\ta)}}{2}}
$$
are two real constants and $w^i(y)$, $i=1,2,3,4$, are  vector fields in $\mathbb{R}^4_2$, depending only on $y$.

Using the expressions of $\ta$ and $\tb$, we have 
$$
\alpha_{1,2}=\skm\Big(|\nu|\sqrt{B}\pm \frac{B}{\tau}\Big).
$$

Setting  $w_{ij}=\langle w^i(y), w^j(y)\rangle$ and evaluating relations \eqref{norme} at $(0,y)$, we get
\begin{align}
w_{11}+2w_{13}+w_{33}=-\frac{4}{\kappa}, \label{e5} \\
\alpha_1^2w_{22}+2\alpha_1\alpha_2w_{24}+\alpha_2^2w_{44}=\frac{4}{\kappa}\ta, \label{e6}\\
\alpha_1w_{12}+\alpha_2w_{14}+\alpha_1w_{23}+\alpha_2w_{34}=0, \label{e7}\\
\alpha_1^3w_{12}+\alpha_2\alpha_1^2w_{14}+\alpha_1\alpha_2^2w_{23}+\alpha_2^3w_{34}=0, \label{e8}\\
\alpha_1^4w_{11}+2\alpha_1^2\alpha_2^2w_{13}+\alpha_2^4w_{33}=D, \label{e9} \\
\alpha_1^2w_{11}+(\alpha_1^2+\alpha_2^2)w_{13}+\alpha_2^2w_{33}=\frac{4}{\kappa}\ta, \label{e10} \\
\alpha_1^4w_{22}+(\alpha_1^3\alpha_2+\alpha_1\alpha_2^3)w_{24}+\alpha_2^4w_{44}=D, \label{e11}\\
\alpha_1^5w_{12}+\alpha_2^3\alpha_1^2w_{14}+\alpha_1^3\alpha_2^2w_{23}+\alpha_2^5w_{34}=0, \label{e12}\\
\alpha_1^3w_{12}+\alpha_2^3w_{14}+\alpha_1^3w_{23}+\alpha_2^3w_{34}=0, \label{e13}\\
\alpha_1^6w_{22}+2\alpha_1^3\alpha_2^3w_{24}+\alpha_2^6w_{44}=E \label{e14}.
\end{align}
From \eqref{e7}, \eqref{e8}, \eqref{e12} and \eqref{e13}, it follows that
$$
w_{12}=w_{14}=w_{23}=w_{34}=0.
$$
Moreover, \eqref{e5}, \eqref{e9} and \eqref{e10} yield
$$
w_{11}=\frac{4  \tau}{\kappa^\frac{3}{2} B}\alpha_2, \qquad 
 w_{13}=0,  \qquad 
w_{33}=-\frac{4   \tau}{\kappa^\frac{3}{2} B}\alpha_1.
$$
Next, by \eqref{e6}, \eqref{e11} and \eqref{e14} we get
$$
w_{22}=w_{11}>0, \qquad
 w_{24}=0, \qquad
w_{44}=w_{33}<0.
$$
\end{proof}
We now prove the following result.

\begin{thm}[of characterization for $B>0$]\label{teo-principal1}
Let $M$ be a helix surface in $\adsl \subset \Rqd$ with constant angle function  $\nu$ so that $B>0$. Then, locally, the position vector of $M$ in $\R_4^2$, with respect to the local coordinates $(u, v)$ on $M$ defined in \eqref{localsystem}, is
$$
F(x,y)=A(y)\,\gamma(x),
$$
where
$$
\gamma(x)=(\sqrt{w_{11}}\,\cos (\alpha_1\, x), -\lambda\sqrt{w_{11}}\,\sin (\alpha_1\, x),\sqrt{-w_{33}}\,\cos (\alpha_2\, x),\lambda \sqrt{-w_{33}}\,\sin (\alpha_2\, x)),
$$
is a twisted geodesic of the Lorentz torus $\mathbb{S}^1(\sqrt{w_{11}})\times \mathbb{S}^1(\sqrt{-w_{33}})  \subset \adsl $, $w_{11}$, $w_{33}$, $\alpha_1$, $\alpha_2$ are the four constants given in Theorem~\ref{t1}, and $A(y)=A(\xi,\xi_1,\xi_2,\xi_3)(y)$ is a 1-parameter family of $4\times 4$ pseudo-orthogonal matrices commuting with $J_1$, as described in \eqref{J_1}, where $\xi$ is a constant and
\begin{align}\label{eq-alpha123}
\cosh^2(\xi_1(y))\,\xi_2'(y)+\sinh^2(\xi_1(y))\,\xi_3'(y)=0.
\end{align}

Conversely, a parametrization $F(x,y)=A(y)\,\gamma(x)$, with $\gamma(x)$ and $A(y)$ as above, defines a  helix surface in $\adsl$ with constant angle function $\nu$.
\end{thm}
\begin{proof}
With respect to the local coordinates $(x,y)$ on $M$ defined in \eqref{localsystem}, Theorem \ref{t1} implies that the position vector of the helix  surface in $\R^4_2$ is given by
$$
F(x,y)=\cos(\alpha_1 x)w^1(y)+\sin(\alpha_1 x)w^2(y)+\cos(\alpha_2 x)w^3(y)+\sin(\alpha_2 x)w^4(y),
$$
where the vector fields  $\{w^i(y)\}_{i=1}^4$ are mutually orthogonal and
$$
\begin{array}{l}
||w^1(y)||=||w^2(y)||=\sluu=\text{constant}, \\[5pt]
||w^3(y)||=||w^4(y)||=\sltt=\text{constant}.
\end{array}
$$
Thus, if we put $e_i(y)=w^i(y)/||w^i(y)||$, $i\in\{1,\dots,4\}$, we can write:
\begin{equation}\label{eq:Fei}
F(x,y) = \sluu\,(\cos(\alpha_1 x)e_1(y)+\sin(\alpha_1 x)e_2(y))+\sltt\,(\cos(\alpha_2 x)e_3(y)+\sin(\alpha_2 x)e_4(y)).
\end{equation}

The identities \eqref{IFnorme}, evaluated in $(0,y)$,  yield:
\begin{equation}\label{eq1bis}\begin{aligned}
    &\alpha_1\,w_{11}\langle J_1e_1,e_2\rangle-\alpha_2 w_{33}\langle J_1e_3,e_4\rangle +\sqrt{-w_{11}w_{33}}\,
    (\alpha_1\langle J_1e_3,e_2\rangle+\alpha_2\langle J_1e_1,e_4\rangle)
    =-\dfrac{2\lm(\lm+\nq)}{\tau \sk},
    \end{aligned}
\end{equation}
\begin{equation}
     \langle J_1e_1,e_3\rangle=0,
\end{equation}
\begin{equation}\label{39}\begin{aligned}
w_{11}\alpha_1^3\langle J_1e_1,e_2\rangle+\sqrt{-w_{11}w_{33}}(\alpha_1^2\alpha_2 \langle J_1e_1,e_4\rangle+\alpha_2^2\alpha_1 \langle J_1e_3,e_2\rangle)-w_{33}\alpha_2^3\langle J_1e_3,e_4\rangle=-L,
\end{aligned}
\end{equation}
\begin{equation}
    \langle J_1e_2,e_4\rangle=0,
\end{equation}
\begin{equation}\label{eq2bis}
    \alpha_1\langle J_1e_2,e_3\rangle+\alpha_2\langle J_1e_1,e_4\rangle=0,
\end{equation}
\begin{equation}\label{eq3bis}
    \alpha_2\langle J_1 e_2,e_3\rangle+\alpha_1\langle J_1e_1,e_4\rangle=0.
\end{equation}
We point out that  to obtain the previous identities we have divided by $\alpha_1^2-\alpha_2^2=\kappa \tau^{-1}|\nu|\sqrt{B^3}$, which,  by the assumption on $\nu$, is always different from zero.
From \eqref{eq2bis} and \eqref{eq3bis}, taking into account  $\alpha_1^2-\alpha_2^2\neq 0$, we find
\begin{equation}\label{eq4bis}
     \langle J_1e_3,e_2\rangle=0,\qquad \langle J_1e_1,e_4\rangle=0.
\end{equation}
Therefore,
$$
|\langle J_1 e_1,e_2\rangle|=1=|\langle J_1 e_3,e_4\rangle|.
$$
Substituting \eqref{eq4bis} in \eqref{eq1bis} and \eqref{39}, we obtain the system
\begin{equation}\label{j34}
\left\{\begin{aligned}\nonumber
&\alpha_1\, w_{11}\langle J_1e_1,e_2\rangle-\alpha_2\,w_{33}\langle J_1e_3,e_4\rangle=-\dfrac{2\lm(\lm+\nq)}{\tau \sk}, \\
 & \alpha_1^3\, w_{11}\langle J_1e_1,e_2\rangle-\alpha_2^3\,w_{33}\langle J_1e_3,e_4\rangle=-L,
\end{aligned}
\right.
\end{equation}
a solution of which is given by
$$
\langle J_1e_1,e_2\rangle=\frac{2\alpha_2^2(1+\lm \nq)-\tau L \sk}{\tau \sk \, w_{11}\,\alpha_1 (\alpha_1^2-\alpha_2^2)},\qquad \langle J_1e_3,e_4\rangle=\frac{2 \alpha_1^2(1+\lm \nq)+\tau L \sk}{\tau  \sk \, w_{33}\,\alpha_2 (\alpha_1^2-\alpha_2^2)}.
$$
Consequently,
$\langle J_1e_1,e_2\rangle=\langle J_1e_3,e_4\rangle=-\lm$ and $J_1e_1=-\lm e_2$, $J_1e_3=\lm e_4$.
Then, if we fix the {pseudo-}orthonormal basis  of $\R^4_2$ given by
$$
\tilde{E}_1=(1,0,0,0),\quad \tilde{E}_2=(0,-\lambda,0,0),\quad \tilde{E}_3=(0,0,1,0),\quad \tilde{E}_4=(0,0,0,\lambda),
$$
there must exists a $1$-parameter family of $4\times 4$ pseudo-orthogonal matrices $A(y)\in \mathrm{O}_2(4)$, with $J_1A(y)=A(y)J_1$,
such that $e_i(y)=A(y)\tilde{E}_i$ for all indices $i=1,\dots,4$..
Replacing $e_i(y)=A(y)\tilde{E}_i$ in \eqref{eq:Fei} we obtain
$$
F(x,y)=A(y)\gamma(x),
$$
where
$$
\gamma(x)=(\sqrt{w_{11}}\,\cos (\alpha_1\, x), -\lambda\sqrt{w_{11}}\,\sin (\alpha_1\, x),\sqrt{-w_{33}}\,\cos (\alpha_2\, x),\lambda \sqrt{-w_{33}}\,\sin (\alpha_2\, x)),
$$
is a twisted geodesic of the Lorentzian torus $\mathbb{S}^1(\sqrt{w_{11}})\times \mathbb{S}^1(\sqrt{-w_{33}})\subset\adsl$. 

Let now examine the $1$-parameter family $A(y)$ that, according to \eqref{eq-descrizione-A},
depends on four functions $\xi_1(y)$, $\xi_2(y)$, $\xi_3(y)$ and $\xi(y)$. From \eqref{localsystem}, it follows that $\langle F_y, F_y\rangle=(\lm+\nq)=\mathrm{constant}$. The latter implies that
\begin{equation}\label{eq-fv-fv-sin-theta-d-u}
\frac{\partial}{\partial x}\langle F_y, F_y\rangle_{| x=0}=0.
\end{equation}
Now, if we denote by ${\mathbf c_1},{\mathbf c_2},{\mathbf c_3},{\mathbf c_4}$ the four colons of $A(y)$, \eqref{eq-fv-fv-sin-theta-d-u} implies that
\begin{equation}\label{sistem-c23-c24}
\begin{cases}
\langle {\mathbf c_2}',{\mathbf c_3}'\rangle=0\\
\langle {\mathbf c_2}',{\mathbf c_4}'\rangle=0,
\end{cases}
\end{equation}
where by $'$ we mean the derivative with respect to $y$.
Replacing in \eqref{sistem-c23-c24} the expressions of the ${\mathbf c_i}$'s as functions of $\xi_1(y),\xi_2(y),\xi_3(y)$ and $\xi(y)$, we obtain
\begin{equation}\label{sistemK-H}
\begin{cases}
\xi'\, {h_1(y)}=0\\
\xi'\, {k_1(y)}=0,
\end{cases}
\end{equation}
where $h_1(y)$ and $k_1(y)$ are two functions such that
$$
h_1^2+k_1^2=4 (\xi_1')^2+\sinh^2(2\xi_1)\, (-\xi'+\xi_2'+\xi_3')^2.
$$
From \eqref{sistemK-H} we have two possibilities:
\begin{itemize}
\item[(i)] $\xi=\cst$;
\item[] or
\item[(ii)] $4 (\xi_1')^2+\sinh^2(2\xi_1)\, (-\xi'+\xi_2'+\xi_3')^2=0$.
\end{itemize}
We will show that case (ii) cannot occurs. More precisely, if (ii) occurs, then the parametrization $F(x,y)=A(y)\gamma(x)$ defines a Hopf tube, that is, the hyperbolic Hopf vector field $E_1$ is tangent to the surface. 
To this end,  we write the unit normal vector field $N$ to the parametrization
$F(x,y)$ as
$$
N=\frac{N_1 E_1+N_2E_2+N_3 E_3}{\sqrt{|-N_1^2+N_2^2+N_3^2|}},
$$
where
$$
\begin{cases}
N_1=g_{\tau}(F_x,E_3)g_{\tau}(F_y,E_2)-g_{\tau}(F_x,E_2)g_{\tau}(F_y,E_3),\\
N_2=g_{\tau}(F_x,E_1)g_{\tau}(F_y,E_3)-g_{\tau}(F_x,E_3)g_{\tau}(F_y,E_1),\\
N_3=g_{\tau}(F_x,E_2)g_{\tau}(F_y,E_1)-g_{\tau}(F_x,E_1)g_{\tau}(F_y,E_2).
\end{cases}
$$
A long computation then gives
$$\begin{aligned}
N_1&=1/2 (\alpha_1 + \alpha_2) \sqrt{w_{11}} \sqrt{-w_{33}}\,[-2\xi_1'\,\cos(\alpha_1 x+\alpha_2 x+\xi_2-\xi_3)\\&+
  \sinh(2\xi_1)\sin(\alpha_1 x+\alpha_2 x+\xi_2-\xi_3)( -\xi'+\xi_2'+\xi_3')].
  \end{aligned}
$$
Now, case (ii) occurs if and only if either $\xi_1=\mathrm{constant}=0$ or $\xi_1=\cst\neq 0$ and $-\xi'+\xi_2'+\xi_3'=0$. In both cases, we conclude that $N_1=0$ and this implies that $$g_\tau(N, J_1F)=\dfrac{2 \tau}{\sk}g_{\tau}(N,E_1)=0,$$ i.e. the Hopf vector field is tangent to the surface. \\
Thus, we are left with the case where $\xi=\cst$. In this case, \eqref{548} is equivalent to
$$
\cosh^2(\xi_1(y))\,\xi_2'(y)+\sinh^2(\xi_1(y))\,\xi_3'(y)=0
$$
whence we conclude that condition \eqref{eq-alpha123} is satisfied.

The converse of the theorem follows immediately from Proposition \ref{prop-vice}, since a direct
calculation shows that $g_{\tau}(F_x,F_x)=g_{\tau}(E_1,F_x)=-\lm(\lm+\nu^2)$ (and so, \eqref{547} holds),
while \eqref{eq-alpha123} is equivalent to \eqref{548}.
\end{proof}

\begin{cor}\label{cor-2}
Let $M$ be a helix spacelike (respectively, timelike) surface in  $\adsl \subset \Rqd$ with constant angle function $\nu$ such that $B>0$. Then, there exist local coordinates on $M$ such that the position vector of $M$ in $\r^4_2$  is given by
$$
F(s,y)=A(y)\,\gamma(s),
$$
where
$$
\gamma(s)=\frac{2}{\sk}\frac{1}{\sqrt{d^2-1}}\Big(\cos \Big(\frac{\sqrt{\kappa}}{2}\,d\,s\Big),-\lambda\sin \Big(\frac{\sqrt{\kappa}}{2}\,d\,s\Big),d\, \cos \Big(\frac{\sqrt{\kappa}}{2}\frac{s}{d}\Big),\lambda\,d\,\sin \Big(\frac{\sqrt{\kappa}}{2}\frac{s}{d}\Big)\Big)
$$
is a twisted geodesic in the Lorentz torus $\s^1\left(\frac{2}{\sqrt{\kappa}}\frac{1}{\sqrt{d^2-1}}\right)\times\s^1\left(\frac{2}{\sqrt{\kappa}}\frac{d}{\sqrt{d^2-1}}\right)\subset \adsl$ parametrized by arc length, whose slope is given by
$$
d=\frac{\sqrt{B}+\tau\,|\nu|}{\sqrt{\lambda+\nu^2}}.
$$
In addition, $A(y)=A(\xi,\xi_1,\xi_2,\xi_3)(y)$ is a $1$-parameter family of $4\times 4$ pseudo-orthogonal matrices commuting with $J_1$, as described in \eqref{eq-descrizione-A}, with $\xi=\cst$ and
$$
 \cosh^2(\xi_1(y))\xi_2'(y)+\sinh^2(\xi_1(y))\xi_3'(y)=0.
$$
Conversely, a parametrization $F(s,y)=A(y)\,\gamma(s)$, with $\gamma(s)$ and $A(y)$ as above, defines
 a helix surface in $\adsl$.
\end{cor}
\begin{proof}
We consider the curve $\gamma(x)$ given in the Theorem~\ref{teo-principal1}. 
Since $\langle\gamma'(x),\gamma'(x)\rangle=\dfrac{4}{\kappa}\alpha_1\,\alpha_2$, considering $$d:=\sqrt{\frac{\alpha_1}{\alpha_2}}=\frac{\sqrt{B}+\tau\,|\nu|}{\sqrt{\lambda+\nu^2}},$$ from equation~\eqref{w11andw33} and taking into account the equation \eqref{e5} with $w_{13}=0$, we get
$$w_{11}=\frac{4}{\kappa}\frac{1}{d^2-1},\qquad w_{33}=-\frac{4}{\kappa}\frac{d^2}{d^2-1}.$$ Observe that $d>1$. Therefore,
we can consider the arc length reparameterization  of the curve $\gamma$ given by
$$
\gamma(s)=\frac{2}{\sqrt{\kappa}}\frac{1}{\sqrt{d^2-1}}\Big(\cos \Big(\frac{\sqrt{\kappa}}{2}\,d\,s\Big),-\lambda\sin \Big(\frac{\sqrt{\kappa}}{2}\,d\,s\Big),d\, \cos \Big(\frac{\sqrt{\kappa}}{2}\frac{s}{d}\Big),\lambda\,d\,\sin \Big(\frac{\sqrt{\kappa}}{2}\frac{s}{d}\Big)\Big).
$$
Finally, we observe that $d$ represents the slope of the geodesic $\gamma$.
\end{proof}

\subsection{Helix surfaces of $\adsl$ in the case $B=0$}
Integrating \eqref{eqquarta1} and taking into account $\partial_xF=T$,  we prove at once the following.

\begin{prop}\label{t2}
Let $M$ be a helix surface in $\adsl\subset\mathbb{R}^4_2$ with constant angle function $\nu$ such that $B=0$. Then, with respect to the local coordinates $(x,y)$ defined in \eqref{localsystem}, the position vector $F$ of $M$ in $\mathbb{R}^4_2$ is given by
$$
F(x,y)=T(y)x+w(y),
$$
where $w(y)$ is a timelike unit vector field in $\mathbb{R}^4_2$, depending only on $y$.
\end{prop} 

We are now ready to prove the following result.
\begin{thm}[of characterization for $B=0$]\label{tprincipalB0}
Let $M$ be a helix surface in $\adsl\subset\mathbb{R}^4_2$ with constant angle function $\nu$ such that $B=0$. Then, with respect to the local coordinates $(x,y)$ defined in \eqref{localsystem}, the position vector $F$ of $M$ in $\mathbb{R}^4_2$ is given by
$$
F(x,y)=A(y)\gamma(x),
$$
where
$$
\gamma(x)=\left(\nu^2\tau x,0,\dfrac{2}{\sqrt{\kappa}},\nu^2\tau\lambda x\right)
$$
is a straight line of $\adsl$ (contained in the  plane $x_2=x_3-\frac{2}{\sk}=0$) and $A(y)=A(\xi,\xi_1,\xi_2,\xi_3)(y)$ is a $1$-parameter family of $4\times4$ indefinite orthogonal matrices commuting with $J_1$ as described in \eqref{eq-descrizione-A},
with
\begin{equation}\begin{aligned}\label{eq-alpha123-b-0}
 &[\xi_2'(y)+\xi_3'(y)-\xi'(y)]\,\sin(\xi_2(y)-\xi_3(y))\,\sinh(2\xi_1(y))\\&-2\, \lambda(\xi'(y)-\xi_2'(y))\,\cosh ^2\xi_1(y)+2\,[\xi_1'(y)\,\cos(\xi_2(y)-\xi_3(y))+\lambda \, \xi_3'(y)\,\sinh^2\xi_1(y)]=0.
  \end{aligned}
\end{equation}
Conversely, a parametrization
$$
F(x,y)=A(y)\left(\nu^2\tau x,0,\dfrac{2}{\sqrt{\kappa}},\nu^2\tau\lambda x\right),
$$
with $A(y)$ as above, defines a helix surface in the anti-De Sitter space $\adsl$ with constant angle  function $\nu$.
\end{thm} 

\begin{proof}
From Proposition~\ref{t2} we say that
\begin{align}\label{Vx}
F(x,y)=T(y)x+w(y),
\end{align}
where $w(y)$ is a vector field in $\mathbb{R}^4_2$, depending only on $y$. 
Using \eqref{Vx} and evaluating the first three equations of \eqref{norme} and the second equation of \eqref{IFnorme} at $(0,y)$, we get the following identities:
\begin{equation}\label{scprod}
\begin{array}{ll}
    \langle F, F \rangle=\langle w(y),w(y)\rangle=-\dfrac{4}{\kappa},    \qquad &   \langle F_x, F_x \rangle=\langle T(y),T(y)\rangle=0,\\  
    \langle F, F_x \rangle=\langle w(y),T(y) \rangle=0,    \qquad &   \langle J_1w, T \rangle=-\dfrac{2\lambda (\lambda+\nu^2)}{\tau \sqrt{\kappa}}.    
    \end{array}
\end{equation}
Moreover, evaluating \eqref{TdxF} in $(0,y)$, setting 
$$
G(0,y)=-\nu \cos \varphi\, E_{2|F(0,y)}-\nu \sin \varphi\, E_{3|F(0,y)} .
$$
and using  \eqref{scprod}, we have
$$
\begin{array}{ll}
    \langle J_1w, G(0,y) \rangle=0,    \quad &   \langle G(0,y), G(0,y) \rangle=\nu^2 . 
    \end{array}
$$
In particular, setting
$$
g^1(y)=\frac{1}{\vert \nu\vert}G(0,y), \qquad 
g^3(y)=\frac{\sqrt{\kappa}}{2}w(y),
$$
we have that $\{g^1(y),J_1g^1(y),g^3(y),J_1g^3(y)\}$ is a pseudo-orthonormal basis of $\mathbb{R}^4_2$.
Consequently, if we fix the orthonormal basis $\{\Hat E_i\}^4_{i=1}$ of $\mathbb{R}^4_2$ given by 
$$\Hat E_1=(1,0,0,0),\qquad \Hat E_2=(0,1,0,0), \qquad \Hat E_3=(0,0,1,0),\qquad  \Hat E_4=(0,0,0,1),$$
there exists a $1$-parameter family of matrices $A(y)\in O_2(4)$, with $J_1A(y)=A(y)J_1$ such that
$$g^1(y)=A(y)\Hat E_1,\qquad J_1g^1(y)=A(y)\Hat E_2, \qquad g^3(y)=A(y)\Hat E_3,\qquad  J_1g^3(y)=A(y)\Hat E_4.$$
Then, \eqref{Vx} becomes
$$
F(x,y)=\frac{2}{\sqrt{\kappa}}g^3(y)+\nu^2\tau x \left( g^1(y)+\lambda J_1g^3(y)\right)=A(y)\left(\nu^2\tau x,0,\dfrac{2}{\sqrt{\kappa}},\nu^2\tau\lambda x\right).
$$
Finally, according to \eqref{eq-descrizione-A}, the $1$-parameter family $A(y)$ depends on four functions $\xi_1(y)$, $\xi_2(y)$, $\xi_3(y)$ and $\xi(y)$ and, in this case, condition \eqref{548} reduces to
$\langle F_u, F_v\rangle=0$ which is equivalent to \eqref{eq-alpha123-b-0}.

In order to prove  the converse, let
$$
F(x,y)=A(y)\left(\nu^2\tau x,0,\dfrac{2}{\sqrt{\kappa}},\nu^2\tau\lambda x\right) 
$$
be a parametrization, where $A(y)=A(\xi(y),\xi_1(y),\xi_2(y),\xi_3(y))$ is a $1$-parameter family of pseudo-orthogonal matrices
with functions  $\xi(y),\xi_1(y),\xi_2(y),\xi_3(y)$ satisfying \eqref{eq-alpha123-b-0}. Since $A(y)$ satisfies \eqref{eq-alpha123-b-0},  $F$ satisfies \eqref{548}. Thus, in virtue of Proposition~\ref{prop-vice}, we only have to show that \eqref{547} is satisfied. We put
$$
\gamma(x)=\left(\nu^2\tau x,0,\dfrac{2}{\sqrt{\kappa}},\nu^2\tau\lambda x\right).
$$
Now, using \eqref{31} and taking into account the fact that $A(y)$ commutes with $J_1$, we get
$$
 g_{\tau}(F_x,F_x)=(1-\tau^2)\nu^4\tau^2=-\lm(\lm+\nq)
$$
and similarly,
$$
 g_{\tau}(E_1,F_x)= {\tau} \langle X_1, F_x \rangle =-\lm \nq \tq =-\lambda(\lambda + \nu^2),
$$
which ends the proof. 
\end{proof}
\subsection{Helix surfaces of $\adsl$ in the case $B<0$}
In this case, we start from \eqref{eqpv1} and prove the following result.
\begin{prop}\label{t3}
Let $M$ be a helix surface in $\adsl$ with constant angle function $\nu$ such that $B<0$. Then, with respect to the local coordinates $(x,y)$ defined above, the position vector $F$ of $M$ in $\mathbb{R}^4_2$ is given by
$$
F(x,y)=\cos(\alpha\, x)\,[\cosh (\beta\, x)\,w^1(y)+\sinh (\beta\, x)\,w^3(y)]+\sin(\alpha x)\,[\cosh (\beta\, x)\,w^2(y)+\sinh (\beta\, x)\,w^4(y)],
$$
where
$$
\alpha=-\skm\frac{B}{\lm\tau},\qquad \beta=|\nu|\frac{\sqrt{-\kappa B}}{2},  
$$
are real constants and $w^i(y)$, $i=1,2,3,4$, are linearly independent vector fields in $\mathbb{R}^4_2$, depending only on $y$, such that:
\begin{align}\label{wijbmin0}
w_{11}=w_{22}=-w_{33}=-w_{44}=-\dfrac{4}{\kappa}, \qquad w_{14}=-w_{23}=\dfrac{4 \lm |\nu|\tau}{\kappa \sqrt{-B}}.
\end{align}
\end{prop}
\begin{proof}
As $B <0$, we have $\tb^2+4\ta<0$. Integrating equation \eqref{eqpv1}, we obtain
$$
F(x,y)=\cos(\alpha\, x)[\cosh (\beta\, x)\,w^1(y)+\sinh (\beta\, x)\,w^3(y)]+
\sin(\alpha\, x)[\cosh (\beta\, x)\,w^2(y)+\sinh (\beta\, x)\,w^4(y)],
$$
where
$$
\alpha=\frac{\tb}{2}, \qquad \beta=\dfrac{1}{2}\sqrt{-(\tb^2+4\ta)} 
$$
are real constants and $w^i(y)$, $i=1,2,3,4$, are  vector fields in $\mathbb{R}^4_2$, depending only on $y$.  Moreover, using the definition of $\ta$ and $\tb$, we get
$$
\alpha=-\skm\frac{B}{\lm\tau},\qquad \beta=|\nu|\frac{\sqrt{-\kappa B}}{2}.  
$$
Defining $w_{ij}=\langle w^i(y), w^j(y)\rangle$, for all indices $i,j$ and evaluating the relations \eqref{norme} in $(0,y)$, we find:
\begin{align}
\label{uno2}
    w_{11}=-\dfrac{4}{\kappa},\\
\label{due2}
    \alpha^2\,w_{22}+\beta^2\,w_{33}+2\alpha\,\beta\,w_{23}=\dfrac{4}{\kappa}\tilde{a},\\
\label{tre2}
    \alpha\,w_{12}+\beta \,w_{13}=0,\\
\label{quatro2}
      \alpha\,\Big(\beta^2-\alpha^2\Big)\,w_{12}+2\alpha\,\beta^2\,w_{34}+2\alpha^2\,\beta\,w_{24}
     +\beta\,\Big(\beta^2-\alpha^2\Big)\,w_{13}=0,\\
\label{cinque2}
   \Big(\beta^2-\alpha^2\Big)^2\,w_{11}+4\alpha^2\beta^2\,w_{44}
     +4\alpha\,\beta\,\Big(\beta^2-\alpha^2\Big)\,w_{14}=D,\\
\label{sei2}
   \Big(\beta^2-\alpha^2\Big)\,w_{11}+2 \alpha\,\beta\,w_{14}=-\dfrac{4}{\kappa}\tilde{a},\\
\label{sette2}
    \alpha^2\, \Big(3\beta^2-\alpha^2\Big)\,w_{22}+\beta^2 \,\Big(\beta^2-3\alpha^2\Big)\,w_{33}
     +4\alpha\,\beta\,(\beta^2-\alpha^2)\,w_{23}=-D,\\
\nonumber
   \alpha\, \Big(3\beta^2-\alpha^2\Big)\, \Big(\beta^2-\alpha^2\Big)\,w_{12}+2\alpha\,\beta^2 \,\Big(\beta^2-3\alpha^2\Big)\,w_{34}\\
     +\beta\,\Big(\beta^2-3\alpha^2\Big)\,\Big(\beta^2-\alpha^2\Big)w_{13}+2\alpha^2\,\beta\, \Big(3\beta^2-\alpha^2\Big)\,w_{24}=0,\label{otto2}\\
\label{nove2}
   \alpha\,\Big(3\beta^2-\alpha^2\Big)\,w_{12}+\beta\,\Big(\beta^2-3\alpha^2\Big)\,w_{13}
  =0,\\
\label{dieci2}
   \alpha^2\, \Big(3\beta^2-\alpha^2\Big)^2\,w_{22}+\beta^2 \,\Big(\beta^2-3\alpha^2\Big)^2\,w_{33}+2\alpha\,\beta\,\Big(3\beta^2-\alpha^2\Big)\, \Big(\beta^2-3\alpha^2\Big)\,w_{23}=E.
\end{align}

From  \eqref{uno2}, \eqref{cinque2} and \eqref{sei2}, it follows that
$$w_{11}=-w_{44}=-\dfrac{4}{\kappa}, \qquad w_{14}=\frac{4\beta}{\kappa\alpha}=\dfrac{4 \lm |\nu|\tau}{\kappa \sqrt{-B}}.$$
Also, from \eqref{tre2} and \eqref{nove2}, we obtain
$$w_{12}=w_{13}=0$$ and, therefore, from \eqref{quatro2} and \eqref{otto2},
$$w_{24}=w_{34}=0.$$
Moreover, using \eqref{due2}, \eqref{sette2} and \eqref{dieci2}, we get
$$w_{22}=-w_{33}=-\dfrac{4}{\kappa}, \qquad w_{23}=-\frac{4\beta}{\kappa\alpha}=-\dfrac{4 \lm |\nu|\tau}{\kappa \sqrt{-B}}.$$
\end{proof}
We now prove the following.
\begin{thm}[of characterization for $B<0$]\label{teo-principal2}
Let $M$ be a helix surface in  $\adsl$ with constant angle function $\nu$ so that $B<0$. Then, locally, the position vector of $M$ in $\R^4_2$, with respect to the local coordinates $(x,y)$ on $M$ defined in \eqref{localsystem}, is given by
$$
F(x,y)=A(y)\,\gamma(x),
$$
where the curve
$
\gamma(x)=(\gamma_1(x),\gamma_2(x),\gamma_3(x),\gamma_4(x))
$
has components
$$\left\{\begin{aligned}
\gamma_1(x)&=\frac{2\rad}{\sqrt{-\kappa B}}\,\cos(\alpha x)\,\sinh (\beta\,x),\\
\gamma_2(x)&=\frac{2\rad}{\sqrt{-\kappa B}}\,\sin(\alpha x)\,\sinh (\beta\,x),\\
\gamma_3(x)&=\dfrac{2}{\sk}\cos(\alpha x)\,\cosh (\beta\,x)-\frac{2\lm \tau |\nu|}{\sqrt{-\kappa B}}\,\sin(\alpha x)\,\sinh (\beta\,x),\\
\gamma_4(x)&=\dfrac{2}{\sk}\sin(\alpha x)\,\cosh (\beta\,x)+\frac{2\lm \tau |\nu|}{\sqrt{-\kappa B}}\,\cos(\alpha x)\,\sinh (\beta\,x),\\
\end{aligned}
\right.
$$
with
$$
\alpha=-\skm\frac{B}{\lm\tau},\qquad \beta=|\nu|\frac{\sqrt{-\kappa B}}{2},  
$$
and $A(y)=A(\xi,\xi_1,\xi_2,\xi_3)(y)$ is a $1$-parameter family of $4\times 4$ pseudo-orthogonal matrices commuting with $J_1$, as described in \eqref{J_1}, where $\xi$ is a constant and
\begin{equation}\label{eq-alpha123second}
\begin{aligned}
&\vert \nu\vert \sqrt{\lambda+\nu^2}\,[2\cos(\xi_2(y)-\xi_3(y))\,\xi_1'(y)+(\xi_2'(y)+\xi_3'(y))\sin(\xi_2(y)-\xi_3(y))\,\sinh(2\xi_1(y))]\\
&+2 \lambda \tau \nu^2\,[\cosh^2(\xi_1(y)) \,\xi_{2}'(y)+\sinh^2(\xi_1(y))\, \xi_{3}'(y)] =0.
 \end{aligned}
\end{equation}
Conversely, a parametrization $F(x,y)=A(y)\,\gamma(x)$, with $\gamma(x)$ and $A(y)$ as above, defines
 a helix surface in $\adsl$ with constant angle function $\nu\neq 0$.
\end{thm}
\begin{proof}
From \eqref{wijbmin0}, we can define the following  pseudo-orthonormal basis in $\R^4_2$:
$$\left\{\begin{aligned}
e_1(y)&=\frac{\sk}{2\rad}[\sqrt{-B}\,w^3(y)-\lm\tau|\nu|\,w^2(y)],\\
e_2(y)&=\frac{\sk}{2\rad}[\sqrt{-B}\,w^4(y)+\lm\tau|\nu|\,w^1(y)],\\
e_3(y)&=\skm w^1(y),\\
e_4(y)&=\skm w^2(y),
\end{aligned}
\right.
$$
with $\langle e_1,e_1\rangle=1=\langle e_2,e_2\rangle$ and $\langle e_3,e_3\rangle=-1=\langle e_4,e_4\rangle$.
Evaluating the identities~\eqref{IFnorme} in  $(0,y)$, and taking into account that:
$$
\begin{aligned}\nonumber
&F(0,y)=w^1(y),\\
&F_x(0,y)=\alpha\,w^2(y)+\beta\,w^3(y),\\
&F_{xx}(0,y)=\Big(\beta^2-\alpha^2\Big)\,w^1(y)+2\alpha\,\beta\,w^4(y),\\
&F_{xxx}(0,y)=\alpha\,\Big(3\beta^2 -\alpha^2\Big)\,w^2(y)+\beta\,\Big(\beta^2-{3}\alpha^2\Big)\,w^3(y),\\
&F_{xxxx}(0,y)=\Big(\beta^4-6\alpha^2\,\beta^2+\alpha^4\Big)\,w^1(y)+
4\alpha\,\beta\,\Big(\beta^2-\alpha^2\Big)\,w^4(y),
\end{aligned}
$$
we conclude that
$$
\begin{aligned}\nonumber
\langle J_1w^1,w^2\rangle&=-\langle J_1w^3,w^4\rangle=-\dfrac{4}{\kappa},\\
\langle J_1w^3,w^2\rangle&=\langle J_1w^1,w^4\rangle=0,\\
\langle J_1w^2,w^4\rangle&=\langle J_1w^1,w^3\rangle=-\frac{4 \lm \tau\,|\nu|}{\kappa\sqrt{-B}}.
\end{aligned}
$$
We point out that to obtain the previous identities, we divided by $$\alpha^2-\beta^2=\frac{\kappa}{4}\frac{B}{\tq}(\lm+\nq)$$ which is always different from zero.
Then,
$$\langle J_1e_1,e_2\rangle=-\langle J_1e_3,e_4\rangle=1,$$
$$\langle J_1e_1,e_4\rangle=\langle J_1e_1,e_3\rangle=\langle J_1e_2,e_3\rangle=\langle J_1e_2,e_4\rangle=0.$$ Therefore, we have $$J_1e_1=e_2,\qquad J_1e_3=e_4.$$
Consequently, if we consider the pseudo-orthonormal basis $\{\hat{E}_i\}_{i=1}^4$ of $\R^4_2$ given by
$$
\hat{E}_1=(1,0,0,0),\quad \hat{E}_2=(0,1,0,0),\quad \hat{E}_3=(0,0,1,0),\quad \hat{E}_4=(0,0,0,1),
$$
there must exists a $1$-parameter family of matrices $A(y)\in \mathrm{O}_2(4)$, with $J_1A(y)=A(y)J_1$, such that $e_i(y)=A(y)\hat{E}_i$ for all indices $i=1,\dots,4$.
As 
$$
F=\langle F,e_1\rangle\,e_1+\langle F,e_2\rangle\,e_2-\langle F,e_3\rangle\,e_3-\langle F,e_4\rangle\,e_4,
$$ 
computing $\langle F,e_i\rangle $ and substituting  $e_i(y)=A(y)\hat{E}_i$, we obtain that $F(x,y)=A(y)\,\gamma(x),
$ where $\gamma(x)$ is the curve of $\adsl$ described in the statement.  Proceeding as in the proof of Theorem~\ref{teo-principal1}, we now examine the $1$-parameter family $A(y)$ that, according to \eqref{eq-descrizione-A},
depends on four functions $\xi_1(y)$, $\xi_2(y)$, $\xi_3(y)$ and $\xi(y)$.  from $\langle F_y, F_y\rangle=\lambda + \nu^2=\mathrm{constant}$ we have 
\begin{equation}\label{eq-fv-fv-sin-theta-d-u-3}
\frac{\partial}{\partial x}\langle F_y, F_y\rangle_{| x=0}=0.
\end{equation}
If we denote by ${\mathbf c_1},{\mathbf c_2},{\mathbf c_3},{\mathbf c_4}$ the four columns of $A(y)$, equation \eqref{eq-fv-fv-sin-theta-d-u-3} implies that
\begin{equation}\label{sistem-c23-c24-3}
\begin{cases}
\langle {\mathbf c_1}',{\mathbf c_3}'\rangle=0,\\
2\tau \,\vert \nu\vert\,\langle {\mathbf c_2}',{\mathbf c_3}'\rangle+\lambda \, \sqrt{\lambda+\nu^2}\left[\langle {\mathbf c_2}',{\mathbf c_2}'\rangle+\langle {\mathbf c_3}',{\mathbf c_3}'\rangle\right]=0,
\end{cases}
\end{equation}
where $'$  denotes the derivative with respect to $y$.
Replacing in \eqref{sistem-c23-c24-3} the expressions of the ${\mathbf c_i}$'s as functions of $\xi_1(y),\xi_2(y),\xi_3(y)$ and $\xi(y)$, we obtain
\begin{equation}\label{sistemK-H3}
\begin{cases}
\xi'\, h_2(y)=0,\\
\xi'\, k_2(y)=0,
\end{cases}
\end{equation}
where $h_2(y)$ and $k_2(y)$ are given by
\begin{equation*}
\left\{
\begin{aligned}
 {h_2(y)}&=2\sin(\xi_2-\xi_3)\,\xi_1'+( \xi'-\xi_2'-\xi_3')\cos(\xi_2-\xi_3)\,\sinh(2\xi_1),\\
{k_2(y)}&=\tau\,|\nu|\,[( \xi'-\xi_2'-\xi_3')\sin(\xi_2-\xi_3)\,\sinh(2\xi_1)-2\cos(\xi_2-\xi_3)\,\xi_1']\\ &+\lambda\, (\lambda+\nu^2)\,[2\cosh^2(\xi_1) \xi_{2}'+2\sinh^2(\xi_1)\, \xi_{3}'-\xi' \cosh(2\xi_1)].
\end{aligned}
\right.
\end{equation*}
From \eqref{sistemK-H3} we have two possibilities:
\begin{itemize}
\item[(i)] $\xi=\cst$;
\item[] or
\item[(ii)] {$h_2=k_2=0$}.
\end{itemize}
As 
$$
\begin{aligned}
N_1&=g_{\tau}(F_x,E_3)g_{\tau}(F_y,E_2)-g_{\tau}(F_x,E_2)g_{\tau}(F_y,E_3)\\
&=\sqrt{\frac{\nu^2\,(\lambda+\nu^2)}{\kappa}}\,\Bigg[\cosh(2 \tilde{b} x)\, h_2(y)-\frac{\lambda\sinh(2 \tilde{b} x)\,k_2(y)}{\sqrt{-B}}\Bigg],
\end{aligned}
$$
it results that if the case (ii) happens  than the parametrization $F(x,y)=A(y)\gamma(x)$ defines a Hopf tube. Thus, we can assume that $\xi=\cst$ and in this case  \eqref{548} is equivalent to \eqref{eq-alpha123second}.

The converse easily follows  from Proposition~\ref{prop-vice}, since a direct
calculation shows that $g_{\tau}(F_x,F_x)=g_{\tau}(E_1,F_x)=-\lm(\lm+\nq)$ (and so, \eqref{547} holds),
while \eqref{eq-alpha123second} is equivalent to \eqref{548}.
\end{proof}

\section{Characterization of the helix surfaces of $\adsl$ by general helices }
As a consequence of Proposition~\ref{prop-vice} and the characterization Theorems~\ref{teo-principal1},  \ref{tprincipalB0} and 
\ref{teo-principal2}, in the next result we  will prove that the curves used to describe helix surfaces in 
$\adsl$, are general helices with axis the infinitesimal generator of the Hopf fibers.  We recall that a  {\it general helix} is a non-null curve $\alpha$ in a Lorentzian manifold $(N,h)$, admitting a Killing vector
field $V$ of constant length along $\alpha$,  such that the function angle between $V$ and $\alpha'$ is a non-zero constant. We say that $V$ is an axis of the general helix $\alpha$.   We now prove the following. 
\begin{prop}\label{elica}
The curves $\gamma:\r\to\adsl$  used in the Theorems~\ref{teo-principal1},  \ref{tprincipalB0} and 
\ref{teo-principal2} to characterize a constant angle spacelike (respectively, timelike) surface $M$,  are spacelike (respectively, timelike)  general helices in $\adsl$ with axis $E_1$, so that they meet at constant angle the fibers of the Hopf fibration. This angle is the same in all the three cases.
\end{prop}
\begin{proof}
We first observe that in the three cases  the position vector of $M$ has been expressed as
\begin{align*}
F(x,y)=A(y)\gamma(x),
\end{align*}
where $A(y)=A(\xi,\xi_1,\xi_2,\xi_3)(y)$ is a $1$-parameter family of $4\times 4$ pseudo-orthogonal matrices commuting with $J_1$ and $\gamma(x)$ is a curve on $\adsl$. Therefore,  as $F_x=A(y)\gamma'$,  from \eqref{547} we get
\begin{align*}
g_\tau(\gamma',\gamma')=g_\tau(F_x,F_x)=-\lambda\, (\lambda+\nu^2),
\end{align*}
thus we conclude that if $M$ is a spacelike (respectively, timelike) surface, then $\gamma$ is a spacelike (respectively, timelike) curve. In both cases, the above equation yields
$$
\|\gamma'\|_{\tau}=\sqrt{\lambda+\nu^2}.
$$
Moreover,  as $J_1A(y)=A(y)J_1$, we have $${E_1}_{|F}=\frac{\sqrt{k}}{2\tau}J_1 F=\frac{\sqrt{k}}{2\tau} A(y)J_1 \gamma$$ and then,  from \eqref{547}, we obtain
\begin{align*}
g_\tau(\gamma',{E_1}_{|\gamma})&=g_\tau\Big(\gamma', \frac{\sqrt{k}}{2\tau} J_1\gamma\Big)=g_\tau\Big(A(y)\gamma', \frac{\sqrt{k}}{2\tau} A(y) J_1\gamma\Big)\\
&= g_\tau(F_x,{E_1}_{|F})=-\lambda\,(\lambda+\nu^2).
\end{align*}
Therefore, the angle function that $\gamma$ forms with the hyperbolic Hopf vector field is given by
$$
\frac{g_\tau(\gamma',E_1)}{\|\gamma'\|_{\tau}}=-\lambda\,\sqrt{\lambda+\nu^2},
$$
that is, in the three cases described in Theorems \ref{teo-principal1}, \ref{tprincipalB0} and 
\ref{teo-principal2}, $\gamma$ is a general helix, forming the same constant angle with its axis $E_1$.
\end{proof}

\begin{oss}
As we observed in Remark~\ref{rem52}, when $\tau=1$ we get flat helix surfaces in $\ads$ equipped with its standard metric. The results we obtained  are consistent with the ones deduced  in \cite{LO}, under the requirement of constant angle between $N$ and $E_1$. In this case, $B=-\lm $ and so:
\begin{itemize}
\item the case $B>0$ corresponds to Lorentzian helix surfaces considered in \cite{LO}, as $\lm=-1$;
\item the case $B=0$ cannot occur;
\item the case $B<0$ corresponds to  Riemannian helix surfaces considered in \cite{LO}, as $\lm= 1$.
\end{itemize}
\end{oss}

%\begin{oss}
%As we observed in Remark \ref{rem52}, when we take $\tau=1$ we deal with flat helix surfaces in $\ads$. In this context, the case $B=0$ cannot be considered as $B=-\lm=\pm 1 \neq 0$. In addition, restricting the results in \cite{LO} to the ones for helix surface given by the requirement of constant angle between $N$ and $E_1$, we can observe that:
%\begin{itemize}
%\item the case $B=-\lm>0$ corresponds to the section in \cite{LO} where Lorentzian helix surfaces have been considered, as $\lm$

\end{document}